\documentclass[preprint]{elsarticle}
\usepackage{amssymb,amsmath,amsthm,bm}
\usepackage{url}

\usepackage{natbib}
\biboptions{sort}

\newcommand{\R}{\textbf{R}}
\newcommand{\C}{\textbf{C}}
\newcommand{\N}{\textbf{N}}
\newcommand{\subres}{\textrm{S}}
\newcommand{\imu}{\bm{\mathit{i}}}
\newcommand{\rank}{\textrm{\upshape rank}}
\newcommand{\diag}{\textrm{diag}}

\newtheorem{thm}{Theorem}
\newtheorem{prop}{Proposition}
\newtheorem{lem}{Lemma}

\theoremstyle{definition}
\newtheorem{defn}{Defnition}
\newtheorem{alg}{Algorithm}
\newtheorem{exmp}{Example}

\theoremstyle{remark}
\newtheorem{rem}{Remark}

\begin{document}

\begin{frontmatter}

\title{GPGCD: An iterative method for calculating approximate GCD of
  univariate polynomials\tnoteref{label1}}

\tnotetext[label1]{Preliminary versions of this paper have been
  presented at ISSAC'09 (Seoul, Republic of Korea, July 28--31, 2009)
  \cite{ter2009} and The Joint Conference of ASCM-MACIS 2009
  (Fukuoka, Japan, December 14--17, 2009) \cite{ter2009b}.}

\author{Akira Terui}
\address{Faculty of Pure and Applied Sciences\\
  University of Tsukuba\\
  Tsukuba, 305-8571 Japan}
\ead{terui@math.tsukuba.ac.jp}
\ead[url]{http://researchmap.jp/aterui}

\begin{abstract}
  We present an iterative algorithm for calculating approximate
  greatest common divisor (GCD) of univariate polynomials with the
  real or the complex coefficients.  For a given pair of polynomials
  and a degree, our algorithm finds a pair of polynomials which has a
  GCD of the given degree and whose coefficients are perturbed from
  those in the original inputs, making the perturbations as small as
  possible, along with the GCD.  The problem of approximate GCD is
  transfered to a constrained minimization problem, then solved with
  the so-called modified Newton method, which is a generalization of
  the gradient-projection method, by searching the solution
  iteratively.  We demonstrate that, in some test cases, our algorithm
  calculates approximate GCD with perturbations as small as those
  calculated by a method based on the structured total least norm
  (STLN) method and the UVGCD method, while our method runs
  significantly faster than theirs by approximately up to $30$ or $10$
  times, respectively, compared with their implementation.  We also
  show that our algorithm properly handles some ill-conditioned
  polynomials which have a GCD with small or large leading
  coefficient.
\end{abstract}


\end{frontmatter}

\section{Introduction}

For algebraic computations on polynomials and matrices, approximate
algebraic algorithms are attracting more attention than before.  These
algorithms take inputs with some ``noise'' such as polynomials with
floating-point number coefficients with rounding errors, or more
practical errors such as measurement errors, then, with minimal
changes on the inputs, seek a meaningful answer that reflect desired
property of the input, such as a common factor of a given degree.  By
this characteristic, approximate algebraic algorithms are expected to
be applicable to more wide range of problems, especially those to
which exact algebraic algorithms were not applicable.

As an approximate algebraic algorithm, we consider calculating the
approximate greatest common divisor (GCD) of univariate polynomials
with the real or the complex coefficients, such that, for a given pair
of polynomials and a degree $d$, finding a pair of polynomials which
has a GCD of degree $d$ and whose coefficients are perturbations from
those in the original inputs, while making the perturbations as small
as possible, along with the GCD.  This problem has been extensively
studied with various approaches including the Euclidean method on the
polynomial remainder sequence (PRS) (\cite{bec-lab1998b},
\cite{sas-nod89}, \cite{sch1985}), the singular value decomposition
(SVD) of the Sylvester matrix (\cite{cor-gia-tra-wat1995},
\cite{emi-gal-lom1997}), the LU or QR factorization of the Sylvester
and/or B\'ezout matrix or their displacements (\cite{bin-boi2010},
\cite{boi2007}, \cite{cor-wat-zhi2004}, \cite{zar-ma-fai2000},
\cite{zhi2003})%
\footnote{Note that the article by \citet{bin-boi2010} has absence of
  reference to a literature on computation on structured matrices by
  \citet{pan1990}, whereas the dissertation by \citet{boi2007} has no
  such omission.}, Pad\'e approximation (\cite{pan2001b}),
optimization strategies (\cite{che-gal-mou-yak2011},
\cite{chi-cor-cor1998}, \cite{kal-yan-zhi2006},
\cite{kal-yan-zhi2007}, \cite{kar-lak1998}, \cite{zen2011}).
Furthermore, stable methods for ill-conditioned problems have been
discussed (\cite{bin-boi2010}, \cite{cor-wat-zhi2004},
\cite{ohs-sug-tor1997}, \cite{san-sas2007}).

Among methods in the above, we focus our attention on optimization
strategy in this paper, especially iterative method for approaching an
optimal solution, after transferring the approximate GCD problem into
a constrained minimization problem.  Already proposed algorithms
utilize iterative methods including the Levenberg-Marquardt method
(\cite{chi-cor-cor1998}), the Gauss-Newton method (\cite{zen2011}) and
the structured total least norm (STLN) method (\cite{kal-yan-zhi2006},
\cite{kal-yan-zhi2007}).  Among them, STLN-based methods have shown
good performance calculating approximate GCD with sufficiently small
perturbations efficiently.

Here, we utilize the so-called modified Newton method
(\cite{tan1980}), which is a generalization of the gradient-projection
method (\cite{ros1961}), for solving the constrained minimization
problem.  This method has interesting features such that it combines
the \textit{projection} and the \textit{restoration} steps in the
original gradient-projection method, which reduces the number of
solving a linear system.  We demonstrate that our algorithm calculates
approximate GCD with perturbations as small as those calculated by the
STLN-based methods, while our method show significantly better
performance over them in its speed compared with their implementation,
by approximately up to $30$ times.  Furthermore, we also show that our
algorithm can properly handle some ill-conditioned problems such as
those with GCD containing small or large leading coefficient. We call
our algorithm \textit{GPGCD} after the initials of the
\textit{gradient projection} method. 

In this paper, we present the following expansion from the previous
results (\cite{ter2009}, \cite{ter2009b}) as presenting the new
algorithm for monic polynomials to calculate perturbed polynomials
without giving perturbations for the leading coefficients; providing
experiment results for new test polynomials that have been prepared
more carefully and comparison with the UVGCD method (\cite{zen2008})
(in Section~\ref{sec:test-appgcd}); adding more experiments for
comparison of our algorithm with the STLN-based method and the UVGCD
method (in Section~\ref{sec:test-zeng}).

The rest part of this chapter is organized as follows.  In
Section~\ref{sec:formulation}, we transform the approximate GCD
problem into a constrained minimization problem.  In
Section~\ref{sec:gp}, we review the framework of the
gradient-projection method and the modified Newton method.  In
Section~\ref{sec:gpgcd}, we show an algorithm for calculating the
approximate GCD, and discuss issues in the application of the
gradient-projection method or the modified Newton method.  In
Section~\ref{sec:exp}, we demonstrate performance of our algorithm
with experiments.

\section{Formulation of the Approximate GCD Problem}
\label{sec:formulation}

Let $F(x)$ and $G(x)$ be univariate polynomials with the real or the
complex coefficients, given as
\begin{equation}
  \label{eq:FG}
  \begin{split}
    F(x)
    &
    = f_m x^m + f_{m-1} x^{m-1} + \cdots + f_0,
    \\
    G(x) 
    &
    = g_n x^n + g_{n-1} x^{n-1} + \cdots + g_0,
  \end{split}
\end{equation}
with $0<n\le m$.  We permit $F$ and $G$ to be relatively prime in
general.  For a given integer $d$ satisfying $0<d\le n$, let us
calculate a deformation of $F(x)$ and $G(x)$ in the form of
\begin{equation}
  \label{eq:FtildeGtilde}
  \begin{split}
    \tilde{F}(x)
    &
    = F(x) + \varDelta F(x) = H(x)\cdot \bar{F}(x),
    \\
    \tilde{G}(x) 
    &
    = G(x) + \varDelta G(x) = H(x)\cdot \bar{G}(x),
  \end{split}
\end{equation}
where $\varDelta F(x)$, $\varDelta G(x)$ are polynomials whose degrees
do not exceed those of $F(x)$ and $G(x)$, respectively, $H(x)$ is a
polynomial of degree $d$, and $\bar{F}(x)$ and $\bar{G}(x)$ are
pairwise relatively prime.  If we find $\tilde{F}$, $\tilde{G}$,
$\bar{F}$, $\bar{G}$ and $H$ satisfying \eqref{eq:FtildeGtilde}, then
we call $H$ \textit{an approximate GCD of $F$ and $G$}.  For a given
degree $d$, we tackle the problem of finding an approximate GCD $H$
while minimizing the norm of the deformations $\|\varDelta F(x)\|_2^2 +
\|\varDelta G(x)\|_2^2$.

To make the paper self-contained, we define notations in the theory of
subresultants used below.
\begin{defn}[Sylvester Matrix]
  Let $F$ and $G$ be defined as in \eqref{eq:FG}.  The 
  \emph{Sylvester matrix} of $F$ and $G$, denoted by $N(F,G)$, is an
  $(m+n)\times(m+n)$ matrix constructed from the coefficients of $F$
  and $G$, such that
  \begin{equation*}
    \begin{split}
      N(F,G) &=
      \begin{pmatrix}
        f_m    &        &        & g_n    &        &  \\
        \vdots & \ddots &        & \vdots & \ddots &  \\
        f_0    &        & f_m    & g_0    &        & g_n \\
               & \ddots & \vdots &        & \ddots & \vdots \\
               &        & f_0    &        &        & g_0
      \end{pmatrix}.
      \\[-4mm]
      &\qquad\; \underbrace{\hspace{20mm}}_{n}
      \hspace{3mm} \underbrace{\hspace{17mm}}_{m}
    \end{split}
  \end{equation*}
\end{defn}
\begin{defn}[Subresultant Matrix]
  \label{def:subresmat}
  Let $F$ and $G$ be defined as in \textup{(\ref{eq:FG})}.  For $0\le
  j<n$, the \emph{$j$-th subresultant matrix} of $F$ and $G$, denoted
  by $N_j(F,G)$, is an $(m+n-j)\times(m+n-2j)$ sub-matrix of $N(F,G)$
  obtained by taking the left $n-j$ columns of coefficients of $F$ and
  the left $m-j$ columns of coefficients of $G$, such that
  \begin{equation}
    \label{eq:subresmat}
    \begin{split}
      N_j(F,G) &=
      \begin{pmatrix}
        f_m    &        &        & g_n    &        &  \\
        \vdots & \ddots &        & \vdots & \ddots &  \\
        f_0    &        & f_m    & g_0    &        & g_n \\
               & \ddots & \vdots &        & \ddots & \vdots \\
               &        & f_0    &        &        & g_0
       \end{pmatrix}.
    \\[-4mm]
    &\qquad\; \underbrace{\hspace{20mm}}_{n-j}
    \hspace{3mm} \underbrace{\hspace{17mm}}_{m-j}
    \end{split}
  \end{equation}
\end{defn}
\begin{defn}[Subresultant]
  Let $F$ and $G$ be defined as in \eqref{eq:FG}.  For $0\le j<n$
  and $k=0,\ldots,j$, let $N_{j,k}=N_{j,k}(F,G)$ be a sub-matrix of
  $N_j(F,G)$ obtained by taking the top $m+n-2j-1$ rows and the
  $(m+n-j-k)$-th row (note that $N_{j,k}(F,G)$ is a square matrix).
  Then, the polynomial
  \begin{equation*}
    \subres_j(F,G)
    =|N_{j,j}|x^j+\cdots+|N_{j,0}|x^0
  \end{equation*}
  is called the \emph{$j$-th subresultant} of $F$ and $G$.
\end{defn}

Now, in the case $\tilde{F}(x)$ and $\tilde{G}(x)$ have a GCD of
degree $d$, then the theory of subresultants tells us that the
$(d-1)$-th subresultant of $\tilde{F}$ and $\tilde{G}$ becomes zero,
namely we have
\[
\subres_{d-1}(\tilde{F},\tilde{G})=0.
\]
Then, the $(d-1)$-th subresultant matrix
$N_{d-1}(\tilde{F},\tilde{G})$ has a kernel of dimension equal to $1$.
Thus, there exist polynomials $A(x),B(x)\in\R[x]$ or $\C[x]$
satisfying
\begin{equation}
  \label{eq:abcond}
  A\tilde{F}+B\tilde{G}=0,
\end{equation}
with $\deg(A)<n-d$ and $\deg(B)<m-d$ and $A(x)$ and $B(x)$ are
relatively prime. Therefore, for the given $F(x)$, $G(x)$ and $d$, our
problem is to find $\varDelta F(x)$, $\varDelta G(x)$, $A(x)$ and
$B(x)$ satisfying Eq.\ \eqref{eq:abcond} while making $\|\varDelta
F\|_2^2+\|\varDelta G\|_2^2$ as small as possible.

\subsection{The Real Coefficient Case}
\label{sec:formulation-real}

Assuming that we have $F(x)$ and $G(x)$ as polynomials with the real
coefficients and find an approximate GCD with the real coefficients
as well, we represent $\tilde{F}(x)$, $\tilde{G}(x)$, $A(x)$ and
$B(x)$ with the real coefficients as
\begin{equation}
  \label{eq:fgab}
  \begin{split}
    \tilde{F}(x) &= \tilde{f}_m x^m+\cdots+\tilde{f}_0x^0, \quad
    \tilde{G}(x) = \tilde{g}_n x^n+\cdots+\tilde{g}_0x^0,
    \\
    A(x) &= a_{n-d}x^{n-d}+\cdots+a_0x^0,\quad
    B(x) = b_{m-d}x^{m-d}+\cdots+b_0x^0,
  \end{split}
\end{equation}
respectively, thus $\|\varDelta F\|_2^2+\|\varDelta G\|_2^2$ and Eq.\
\eqref{eq:abcond} become as
\begin{gather}
  \label{eq:objective0}
  \|\varDelta F\|_2^2+\|\varDelta G\|_2^2
  =
  (\tilde{f}_m-f_m)^2+\cdots+(\tilde{f}_0-f_0)^2
  +
  (\tilde{g}_n-g_n)^2+\cdots+(\tilde{g}_0-g_0)^2,
  \\
  \label{eq:abcond2}
  N_{d-1}(\tilde{F},\tilde{G})\cdot \bm{v}=\bm{0},
\end{gather}
respectively, with $N_j(\tilde{F},\tilde{G})$ as in
\eqref{eq:subresmat} and 
\begin{equation}
  \label{eq:abcond2-v}
  \bm{v}={}^t(a_{n-d},\ldots,a_0,b_{m-d},\ldots,b_0).
\end{equation}
Then, Eq.\ \eqref{eq:abcond2} is regarded as a system
of $m+n-d+1$ equations in
$\tilde{f}_m,\ldots,\tilde{f}_0$, $\tilde{g}_n,\ldots,\tilde{g}_0$,
$a_{n-d},\ldots,a_0$, $b_{m-d},\ldots,b_0$, as
\begin{equation}
  \label{eq:constraint0}
  \begin{array}{ccc}
    q_1 = \tilde{f}_m a_{n-d} + \tilde{g}_n b_{m-d} = 0,
    \cdots,
    q_{m+n-d+1} = \tilde{f}_0 a_0 + \tilde{g}_0 b_0 = 0,
  \end{array}
\end{equation}
by putting $q_j$ as the $j$-th row.  Furthermore, for solving the
problem below stably, we add another constraint enforcing
the coefficients of $A(x)$ and $B(x)$ such that
$\|A(x)\|_2^2+\|B(x)\|_2^2=1$; thus we add
\begin{equation}
  \label{eq:constraint1}
  q_0 = a_{n-d}^2+\cdots+a_0^2+b_{m-d}^2+\cdots+b_0^2-1=0
\end{equation}
into Eq.\ \eqref{eq:constraint0}.

Now, we substitute the variables 
\begin{equation}
  \label{eq:var0}
  (\tilde{f}_m,\ldots,\tilde{f}_0, \tilde{g}_n,\ldots,\tilde{g}_0,
  a_{n-d},\ldots,a_0, b_{m-d},\ldots,b_0)
\end{equation}
as $\bm{x}=(x_1,\ldots,x_{2(m+n-d+2)})$, thus Eq.\
\eqref{eq:objective0} and \eqref{eq:constraint0} with
\eqref{eq:constraint1} become
\begin{multline}
  \label{eq:objective1}
  f(\bm{x})=
  (x_1-f_m)^2+
  \cdots+(x_{m+1}-f_0)^2
  \\
  +
  (x_{m+2}-g_n)^2+
  \cdots+(x_{m+n+2}-g_0)^2,
\end{multline}
\begin{equation}
  \label{eq:constraint2}
  \bm{q}(\bm{x})=
  {}^t(q_0(\bm{x}), q_1(\bm{x}), \ldots, q_{m+n-d+1}(\bm{x}))
  =
  \bm{0},
\end{equation}
respectively.  Therefore, the problem of finding an approximate GCD
can be formulated as a constrained minimization problem of finding a
minimizer of the objective function $f(\bm{x})$ in 
\eqref{eq:objective1}, subject to $\bm{q}(\bm{x})=\bm{0}$ in Eq.\
\eqref{eq:constraint2}. 

\subsection{The Complex Coefficient Case}
\label{sec:formulation-complex}

Now let us assume that we have $F(x)$ and $G(x)$ with the complex
coefficients in general, represented as
\begin{equation*}
  \begin{split}
    F(x) &= (f_{m,1}+f_{m,2}\imu) x^m + \cdots + (f_{0,1}+f_{0,2}\imu), \\
    G(x) &= (g_{n,1}+g_{n,2}\imu) x^n + + \cdots + (g_{0,1}+g_{0,2}\imu),
  \end{split}
\end{equation*}
where $f_{j,1}$, $g_{j,1}$, $f_{j,2}$, $g_{j,2}$ are real numbers;
$f_{j,1}$, and $g_{j,1}$ represent the real parts; $f_{j,2}$,
$g_{j,2}$ represent the imaginary parts, with $\imu$ as the imaginary
unit, and find an approximate GCD with the complex coefficients.
Then, we represent $\tilde{F}(x)$, $\tilde{G}(x)$, $A(x)$ and $B(x)$
with the complex coefficients as
\begin{equation}
  \label{eq:fgab-complex}
  \begin{split}
    \tilde{F}(x) &= (\tilde{f}_{m,1}+\tilde{f}_{m,2}\imu) x^m +\cdots+
    (\tilde{f}_{0,1}+\tilde{f}_{0,2}\imu) x^0, \\
    \tilde{G}(x) &= (\tilde{g}_{n,1}+\tilde{g}_{n,2}\imu) x^n +\cdots+
    (\tilde{g}_0x^0+\tilde{g}_{0,2}\imu) x^0,
    \\
    A(x) &= (a_{n-d,1}+a_{n-d,2}\imu) x^{n-d} +\cdots+
    (a_{0,1}+a_{0,2}\imu)x^0,\\ 
    B(x) &= (b_{m-d,1}+b_{m-d,2}\imu) x^{m-d} +\cdots+
    (b_{0,1}+b_{0,2}\imu)x^0, 
  \end{split}
\end{equation}
respectively, where $\tilde{f}_{j,1}$, $\tilde{f}_{j,2}$,
$\tilde{g}_{j,1}$, $\tilde{g}_{j,2}$, $a_{j,1}$, $a_{j,2}$, $b_{j,1}$,
$b_{j,2}$ are real numbers.

For the objective function, $\|\varDelta F\|_2^2+\|\varDelta G\|_2^2$
becomes as
\begin{equation}
  \label{eq:objective-complex}
  \sum_{j=0}^m[(\tilde{f}_{j,1}-f_{j,1})^2 +
  (\tilde{f}_{j,2}-f_{j,2})^2] +
  \sum_{j=0}^n[(\tilde{g}_{j,1}-g_{j,1})^2 +
  (\tilde{g}_{j,2}-g_{j,2})^2].
\end{equation}

For the constraint, Eq.\ \eqref{eq:abcond} becomes as
\begin{multline}
  \label{eq:abcond2-complex}
  \setlength{\arraycolsep}{1pt}
  \begin{pmatrix}
    \tilde{f}_{m,1}+\tilde{f}_{m,2}\imu & & &
    \tilde{g}_{n,1}+\tilde{g}_{n,2}\imu & &  \\
    \vdots      & \ddots &             & \vdots      & \ddots &  \\
    \tilde{f}_{0,1}+\tilde{f}_{0,2}\imu & &
    \tilde{f}_{m,1}+\tilde{f}_{m,2}\imu &
    \tilde{g}_{0,1}+\tilde{g}_{0,2}\imu & &
    \tilde{g}_{n,1}+\tilde{g}_{n,2}\imu \\
    & \ddots & \vdots &    & \ddots      & \vdots \\
    & & \tilde{f}_{0,1}+\tilde{f}_{0,2}\imu & & &
    \tilde{g}_{0,1}+\tilde{g}_{0,2}\imu 
  \end{pmatrix}
  \\
  \times
  \begin{pmatrix}
    a_{n-d,1}+a_{n-d,2}\imu
    \\
    \vdots
    \\
    a_{0,1}+a_{0,2}\imu
    \\
    b_{m-d,1}+b_{m-d,2}\imu
    \\
    \vdots
    \\
    b_{0,1}+b_{0,2}\imu
  \end{pmatrix}
  =
  \bm{0}
  .
\end{multline}
By expressing the subresultant matrix and the column vector in
\eqref{eq:abcond2-complex} separated into the real and the complex parts,
respectively, we express \eqref{eq:abcond2-complex} as
\begin{equation}
  \label{eq:abcond2-complex-2}
  (N_1+N_2\imu)(\bm{v}_1+\bm{v}_2\imu)=\bm{0},
\end{equation}
with
\begin{equation}
  \label{eq:n1n2v1v2}
  \setlength{\arraycolsep}{1pt}
  \begin{array}{c}
    N_1 =
    \setlength{\arraycolsep}{1pt}
    \begin{pmatrix}
      \tilde{f}_{m,1} & & & \tilde{g}_{n,1} & &  \\
      \vdots      & \ddots &             & \vdots      & \ddots &  \\
      \tilde{f}_{0,1} & & \tilde{f}_{m,1} & \tilde{g}_{0,1} & &
      \tilde{g}_{n,1} \\ 
      & \ddots & \vdots &    & \ddots      & \vdots \\
      & & \tilde{f}_{0,1} & & & \tilde{g}_{0,1}
    \end{pmatrix},
    \;
    N_2 =
    \setlength{\arraycolsep}{1pt}
    \begin{pmatrix}
      \tilde{f}_{m,2} & & & \tilde{g}_{n,2} & &  \\
      \vdots      & \ddots &             & \vdots      & \ddots &  \\
      \tilde{f}_{0,2} & & \tilde{f}_{m,2} & \tilde{g}_{0,2} & &
      \tilde{g}_{n,2} \\ 
      & \ddots & \vdots &    & \ddots      & \vdots \\
      & & \tilde{f}_{0,2} & & & \tilde{g}_{0,2}
    \end{pmatrix},
    \\
    \bm{v}_1 = {}^t(a_{n-d,1},\ldots,a_{0,1},b_{m-d,1},\ldots,b_{0,1}),
    \\
    \bm{v}_2 = {}^t(a_{n-d,2},\ldots,a_{0,2},b_{m-d,2},\ldots,b_{0,2}).
  \end{array}
\end{equation}
We can expand the left-hand-side of Eq.\ \eqref{eq:abcond2-complex-2} as
\begin{equation*}
  (N_1+N_2\imu)(\bm{v}_1+\bm{v}_2\imu)=
  (N_1\bm{v}_1-N_2\bm{v}_2)+\imu(N_1\bm{v}_2+N_2\bm{v}_1),
\end{equation*}
thus, Eq.\ \eqref{eq:abcond2-complex-2} is equivalent to a system of
equations 
\begin{equation*}
  N_1\bm{v}_1-N_2\bm{v}_2=\bm{0},\quad
  N_1\bm{v}_2+N_2\bm{v}_1=\bm{0},
\end{equation*}
which is expressed as
\begin{equation}
  \label{eq:abcond2-complex-3}
  \begin{pmatrix}
    N_1 & -N_2 \\
    N_2 & N_1
  \end{pmatrix}
  \begin{pmatrix}
    \bm{v}_1 \\ \bm{v}_2
  \end{pmatrix}
  =
  \bm{0}.
\end{equation}

Furthermore, as well as in the real coefficients case, we add another
constraint for the coefficient of $A(x)$ and $B(x)$ as
\begin{multline}
  \label{eq:abconstraint-complex}
  \|A(x)\|_2^2+\|B(x)\|_2^2
  =(a_{n-d,1}^2+\cdots+a_{0,1}^2)+(b_{m-d,1}^2+\cdots+b_{0,1}^2)
  \\
  +(a_{n-d,2}^2+\cdots+a_{0,2}^2)+(b_{m-d,2}^2+\cdots+b_{0,2}^2)-1=0,
\end{multline}
which can be expressed together with \eqref{eq:abcond2-complex-3} as
\begin{equation}
  \label{eq:abcond2-complex-4}
  \begin{pmatrix}
    {}^t\bm{v}_1 & {}^t\bm{v}_2 & -1 \\
    N_1 & -N_2 & \bm{0} \\
    N_2 & N_1 & \bm{0}
  \end{pmatrix}
  \begin{pmatrix}
    \bm{v}_1 \\ \bm{v}_2 \\ 1
  \end{pmatrix}
  =
  \bm{0},
\end{equation}
where Eq.\ \eqref{eq:abconstraint-complex} has been put on the top of
Eq.\ \eqref{eq:abcond2-complex-3}.  Note that, in Eq.\
\eqref{eq:abcond2-complex-4}, we have total of $2(m+n-d+1)+1$
equations in the coefficients of polynomials in
\eqref{eq:fgab-complex} as a constraint, with the $j$-th row of which
is expressed as $q_j=0$, as similarly as in the real case
\eqref{eq:constraint0} with \eqref{eq:constraint1}.

Now, as in the real case, we substitute the variables
\begin{multline}
  \label{eq:var0-complex}
  (
  \tilde{f}_{m,1},\ldots,\tilde{f}_{0,1},
  \tilde{g}_{n,1},\ldots,\tilde{g}_{0,1},
  \tilde{f}_{m,2},\ldots,\tilde{f}_{0,2},
  \tilde{g}_{n,2},\ldots,\tilde{g}_{0,2},
  \\
  a_{n-d,1},\ldots,a_{0,1},
  b_{m-d,1},\ldots,b_{0,1},
  a_{n-d,2},\ldots,a_{0,2},
  b_{m-d,2},\ldots,b_{0,2}
  )
\end{multline}
as $\bm{x}=(x_1,\ldots,x_{4(m+n-d+2)})$, thus Eq.\
\eqref{eq:objective-complex} and \eqref{eq:abcond2-complex-4} become
as
\begin{align}
  f(\bm{x})
  =
  &
  (x_1-f_{m,1})^2+\cdots+(x_{m+1}-f_{0,1})^2
  \nonumber
  \\
  &
  +
  (x_{m+2}-g_{n,1})^2+\cdots+(x_{m+n+2}-g_{0,1})^2
  \nonumber
  \\
  &
  +
  (x_{m+n+3}-f_{m,2})^2+\cdots+(x_{2m+n+3}-f_{0,2})^2
  \nonumber
  \\
  &
  +
  (x_{2m+n+4}-g_{n,2})^2+\cdots+(x_{2(m+n+2)}-g_{0,2})^2,
  \label{eq:objective-complex2}
  \\
  \bm{q}(\bm{x})
  =
  &
  \;
  {}^t(q_1(\bm{x}), \ldots, q_{2(m+n-d+1)+1}(\bm{x}))
  =
  \bm{0},
  \label{eq:constraint-complex2}
\end{align}
respectively.  Therefore, the problem of finding an approximate GCD
can be formulated as a constrained minimization problem of finding a
minimizer of the objective function $f(\bm{x})$ in Eq.\
\eqref{eq:objective-complex2}, subject to $\bm{q}(\bm{x})=\bm{0}$ in Eq.\
\eqref{eq:constraint-complex2}.

\section{The Gradient-Projection Method
  and the Modified Newton Method} 
\label{sec:gp}

In this section, we consider the problem of minimizing an objective
function $f(\bm{x}):\R^n\rightarrow\R$, subject to the constraints
$\bm{q}(\bm{x})=\bm{0}$ for
$\bm{q}(\bm{x})={}^t(q_1(\bm{x}),q_2(\bm{x}),\ldots,q_m(\bm{x}))$,
with $m\le n$, where $q_j(\bm{x})$ is a function of
$\R^n\rightarrow\R$, and $f(\bm{x})$ and $q_j(\bm{x})$ are twice
continuously differentiable (here, we refer presentations of the
problem to 
\citet{tan1980} and the references therein).

If we assume that the Jacobian matrix
\[
J_{\bm{q}}(\bm{x})=
\left(
  \frac{\partial q_i}{\partial x_j}
\right)
\]
is of full rank, or
\begin{equation}
  \label{eq:rankj}
  \rank(J_{\bm{q}}(\bm{x}))=m,  
\end{equation}
on the feasible region $V_{\bm{q}}$ defined by
\[
V_{\bm{q}}=\{\bm{x}\in\R^n\mid \bm{q}(\bm{x})=\bm{0}\},
\]
then the feasible region $V_{\bm{q}}$ is an $(n-m)$-dimensional
differential manifold in $\R^n$ and $f$ is differentiable function on
the manifold $V_{\bm{q}}$.  Thus, our problem is to find a point in
$V_{\bm{q}}$, which will be a candidate of a local minimizer,
satisfying the well-known ``first-order necessary conditions'' (for
the proof, refer to the literature on optimization such as
\citet{noc-wri2006}).
\begin{thm}[First-order necessary conditions]
  \label{thm:fonc}
  Suppose that $\bm{x}^*\in V_{\bm{q}}$ is a local solution of the
  problem in the above, that the functions $f(\bm{x})$ and
  $\bm{q}(\bm{x})$ are continuously differentiable at $\bm{x}^*$, and
  that we have \eqref{eq:rankj} at $\bm{x}^*$.  Then, there exist a
  \emph{Lagrange multiplier vector} $\bm{\lambda^*}\in\R^m$ satisfying
  \[
  \nabla f(\bm{x}^*)-{}^t(J_{\bm{q}}(\bm{x}^*))\bm{\lambda}^*=\bm{0},
  \quad
  \bm{q}(\bm{x}^*)=\bm{0}.
  \qed
  \]
\end{thm}

\subsection{The Gradient-Projection Method}

Let $\bm{x}_k\in\R^n$ be a feasible point, or a point satisfying
$\bm{x}_k\in V_{\bm{q}}$.  Rosen's gradient projection method
(\cite{ros1961}) is based on projecting the steepest descent direction
onto the tangent space of the manifold $V_{\bm{q}}$ at $\bm{x}_k$,
which is denoted to $T_{\bm{x}_k}$ and represented by the kernel of
the Jacobian matrix $J_{\bm{q}}(\bm{x}_k)$ as
\begin{equation}
  \label{eq:tx}
  T_{\bm{x}_k} = \ker(J_{\bm{q}}(\bm{x}_k))=\{\bm{z}\in\R^n\mid
  J_{\bm{q}}(\bm{x}_k)\bm{z}=\bm{0}\in\R^m\}.
\end{equation}
We have steepest descent direction of the objective function $f$ at
$\bm{x}_k$ as
\begin{equation}
  \label{eq:-nablaf}
  -\nabla f(\bm{x}_k)=-{}^t
  \left(
    \frac{\partial f}{\partial x_1},
    \ldots,
    \frac{\partial f}{\partial x_n}
  \right)
  .
\end{equation}
Then, the search direction $\bm{d}_k$ is defined by the projection of the
steepest descent direction of $f$ in \eqref{eq:-nablaf} onto
$T_{\bm{x}_k}$ in \eqref{eq:tx} as
\begin{equation}
  \label{eq:projection}
  \bm{d}_k=-P(\bm{x}_k)\nabla f(\bm{x}_k).  
\end{equation}
Here, $P(\bm{x}_k)$ is the orthogonal projection operator on
$T_{\bm{x}_k}$ defined as
\[
P(\bm{x}_k)=I-(J_{\bm{q}}(\bm{x}_k))^+(J_{\bm{q}}(\bm{x}_k)),
\]
where $I$ is the identity matrix and $(J_{\bm{q}}(\bm{x}_k))^+$ is the
Moore-Penrose inverse of $(J_{\bm{q}}(\bm{x}_k))$.  Under the
assumption \eqref{eq:rankj}, we have
\[
(J_{\bm{q}}(\bm{x}_k))^+={}^t(J_{\bm{q}}(\bm{x}_k))\cdot
(J_{\bm{q}}(\bm{x}_k)\cdot{}^t(J_{\bm{q}}(\bm{x}_k)))^{-1} 
\]
(see \citet[Eq.\ (8)]{tan1980}).

With an appropriate step width $\alpha_k$ (see
Remark~\ref{rem:stepwidth}) satisfying $0<\alpha_k\le 1$, let 
\[
\bm{y}_k=\bm{x}_k+\alpha_k\cdot \bm{d}_k.
\]
Since $V_{\bm{q}}$ is nonlinear in general, $\bm{y}_k$ may not in
$V_{\bm{q}}$: in such a case, we take a \textit{restoration} move to
bring $\bm{y}_k$ back to $V_{\bm{q}}$, as follows.  Let $\bm{x}\in\R^n$
be an arbitrary point. Then, at $\bm{y}_k$, the constraint
$\bm{q}(\bm{x})$ can be linearly approximated as
\[
\bm{q}(\bm{y}_k+\bm{x})\simeq \bm{q}(\bm{y}_k)+J_{\bm{q}}(\bm{y}_k)\bm{x}.
\]
Assuming $\bm{y}_k+\bm{x}\in V_{\bm{q}}$, we have
$\bm{q}(\bm{y}_k+\bm{x})=\bm{0}$ thus the approximation of $\bm{x}$
can be calculated as
\begin{equation}
  \label{eq:restore}
  \bm{x}=-(J_{\bm{q}}(\bm{y}_k))^+ \bm{q}(\bm{y}_k).  
\end{equation}
If $\bm{y}_k$ is sufficiently close to $V_{\bm{q}}$, then we can restore
$\bm{y}_k$ back onto $V_{\bm{q}}$ by applying \eqref{eq:restore}
iteratively for several times.  Note that the restoration move can also
be used in the case the initial point of the minimization process is
away from the feasible region $V_{\bm{q}}$.

Summarizing the above, we obtain an algorithm for the gradient
projection as follows.
\begin{alg}[The gradient-projection method (\cite{ros1961})]
  \label{alg:gp}
  \mbox{\\}
  \begin{description}
  \item[Step 1] \textbf{[Restoration]} If the given point $\bm{x}_0$
    does not satisfy $\bm{x}_0\in V_{\bm{q}}$, first move $\bm{x}_0$
    onto $V_{\bm{q}}$ by the iteration of Eq.\ \eqref{eq:restore},
    then let $\bm{x}_0$ be the restored point on $V_{\bm{q}}$.  Let
    $k=0$.
  \item[Step 2] \textbf{[Projection]} For $\bm{x}_k$, calculate
    $\bm{d}_k=-P(\bm{x}_k)\nabla f(\bm{x}_k)$ by
    \eqref{eq:projection}.  If $\|\bm{d}_k\|$ is sufficiently small
    for an appropriate norm, go to Step 4.  Otherwise, calculate the
    step width $\alpha_k$ by an appropriate line search method (see
    Remark~\ref{rem:stepwidth}) then let
    $\bm{y}_{k,0}=\bm{x}_k+\alpha_k\bm{d}_k$.
  \item[Step 3] \textbf{[Restoration]} If $\bm{q}(\bm{y}_{k,0})\ne
    \bm{0}$, move $\bm{y}_{k,0}$ back onto $V_{\bm{q}}$ iteratively by
    \eqref{eq:restore}.  Let
    $\bm{y}_{k,l+1}=\bm{y}_{k,l}-(J_{\bm{q}}(\bm{y}_{k,l}))^+
    \bm{q}(\bm{y}_{k,l})$ for $l=0,1,2,\ldots$.  When $\bm{y}_{k,l}$
    satisfies $\bm{q}(\bm{y}_{k,l})\simeq \bm{0}$, then let
    $\bm{x}_{k+1}=\bm{y}_{k,l}$ and go to Step 2.
  \item[Step 4] \textbf{[Checking the first-order necessary
      conditions]} If $\bm{x}_k$ satisfies Theorem~\ref{thm:fonc},
    then return $\bm{x}_k$.
  \end{description}
\end{alg}
\begin{rem}
  \label{rem:stepwidth}
  Choosing appropriate step width in the iteration is a fundamental
  issue in optimization method and is discussed in standard literature
  of optimization (\textit{e.g.}\ \cite{noc-wri2006}).  Although we
  simply set $\alpha_k=1$ in our implementation, more sophisticated
  calculation of step width might improve accuracy and/or convergence
  of the algorithm (see also concluding remarks
  (Section~\ref{sec:remark})).
\end{rem}

\subsection{The Modified Newton Method}

The modified Newton method by
\citet{tan1980} is a generalization
of the Newton's method, which derives several different methods, by
modifying the Hessian of the Lagrange function.  A generalization of
the gradient-projection method combines the \textit{restoration step}
and the \textit{projection step} in Algorithm~\ref{alg:gp}.  For
$\bm{x}_k\in V_{\bm{q}}$, we calculate the search direction
$\bm{d}_k$, along with the associated Lagrange multipliers
$\bm{\lambda}_{k+1}$, by solving a linear system
\begin{equation}
  \label{eq:modnewton}
  \begin{pmatrix}
    I & -{}^t(J_{\bm{q}}(\bm{x}_k)) \\
    J_{\bm{q}}(\bm{x}_k) & \bm{O}
  \end{pmatrix}
  \begin{pmatrix}
    \bm{d}_k \\ \bm{\lambda}_{k+1}
  \end{pmatrix}
  =
  -
  \begin{pmatrix}
    \nabla f(\bm{x}_k) \\
    \bm{q}(\bm{x}_k)
  \end{pmatrix}
  ,
\end{equation}
then put $\bm{x}_{k+1}=\bm{x}_k+\alpha_k\cdot \bm{d}_k$ with an
appropriate step width $\alpha_k$.  Solving Eq.\ \eqref{eq:modnewton}
under assumption \eqref{eq:rankj}, we have
\begin{equation}
  \label{eq:modnewtonsol}
  \begin{split}
    \bm{d}_k &= -P(\bm{x}_k)\nabla f(\bm{x}_k) 
    -(J_{\bm{q}}(\bm{x}_k))^+\bm{q}(\bm{x}_k),
    \\
    \bm{\lambda}_{k+1} &= {}^t((J_{\bm{q}}(\bm{x}_k))^+) \nabla
    f(\bm{x}_k)
    - (J_{\bm{q}}(\bm{x}_k)\cdot {}^t(J_{\bm{q}}(\bm{x}_k)))^{-1}
    \bm{q}(\bm{x}_k).  
  \end{split}
\end{equation}
Note that, in $\bm{d}_k$ in \eqref{eq:modnewtonsol}, the term
$-P(\bm{x}_k)\nabla f(\bm{x}_k)$ comes from the projection
\eqref{eq:projection}, while another term
$-(J_{\bm{q}}(\bm{x}_k))^+\bm{q}(\bm{x}_k)$ comes from the restoration
\eqref{eq:restore}.  If we have $\bm{x}_k\in V_{\bm{q}}$, the
iteration formula \eqref{eq:modnewton} is equivalent to the projection
\eqref{eq:projection}.  After an iteration, the new estimate
$\bm{x}_{k+1}$ may not satisfy $\bm{x}_{k+1}\in V_{\bm{q}}$: in such a
case, in the next iteration, the point will be pulled back onto
$V_{\bm{q}}$ by the $-(J_{\bm{q}}(\bm{x}_k))^+\bm{q}(\bm{x}_k)$ term.
Therefore, by solving Eq.\ \eqref{eq:modnewton} iteratively, we expect
that the approximations $\bm{x}_k$ moves toward descending direction
of $f$ along with tracing the feasible set $V_{\bm{q}}$.

Summarizing the above, we obtain an algorithm as follows.
\begin{alg}[The modified Newton method (\cite{tan1980})]
  \label{alg:gpnewton}
  \upshape 
  \mbox{\\}
  \begin{description}
  \item[Step 1] \textbf{[Finding a search direction]} For $\bm{x}_k$,
    calculate $\bm{d}_k$ by solving the linear system
    \eqref{eq:modnewton}.  If $\|\bm{d}_k\|$ is sufficiently small, go
    to Step 2.  Otherwise, calculate the step width $\alpha_k$ by an
    appropriate line search method (see Remark~\ref{rem:stepwidth}),
    let $\bm{x}_{k+1}=\bm{x}_k+\alpha_k\bm{d}_k$, then go to Step 1.
  \item[Step 2] \textbf{[Checking the first-order necessary
      conditions]} If $\bm{x}_k$ satisfies Theorem~\ref{thm:fonc} with
    sufficient accuracy, then return $\bm{x}_k$.
  \end{description}
\end{alg}

\section{The Algorithm for Approximate GCD}
\label{sec:gpgcd}

In applying the gradient-projection method or the modified Newton method
to the approximate GCD problem, we discuss issues in the construction
of the algorithm in detail, such as
\begin{itemize}
\item Representation of the Jacobian matrix $J_{\bm{q}}(\bm{x})$
  (Section \ref{sec:jacobianmat}),
\item Stability of the algorithm by certifying that $J_q(\bm{x})$ has
  full rank (Section \ref{sec:jacobianmatrank}),
\item Setting the initial values (Section~\ref{sec:init}),
\item Regarding the minimization problem as the minimum distance
  problem (Section~\ref{sec:mindistance}),
\item Calculating the actual GCD and correcting the coefficients of
  $\tilde{F}$ and $\tilde{G}$ (Section~\ref{sec:correct}),
\end{itemize}
as follows.  After presenting the algorithm, we give a modification
for preserving monicity for the real coefficient case and running time
analysis, and end this section with examples.

\subsection{Representation of the  Jacobian Matrix}
\label{sec:jacobianmat}

For a polynomial $P(x)\in\R[x]$ or $\C[x]$ represented as
\[
P(x) = p_nx^n+\cdots+p_0x^0,
\]
let $C_k(P)$ be a complex $(n+k,k+1)$ matrix defined as
\[
\begin{array}{ccl}
  C_k(P) & = & 
  \begin{pmatrix}
    p_n & & \\
    \vdots & \ddots & \\
    p_0 & & p_n \\
    & \ddots & \vdots \\
    & & p_0
  \end{pmatrix}
  .
  \\[-2mm]
  & & 
  \hspace{3mm}
  \underbrace{\hspace{19mm}}_{k+1}
\end{array}
\]

We show the Jacobian matrix in the real and the complex coefficient
cases, both of which can easily be constructed in every iteration in
Algorithms~\ref{alg:gp} and \ref{alg:gpnewton}.

\subsubsection{The Real Coefficient Case}

For co-factors $A(x)$ and $B(x)$ as in \eqref{eq:fgab}, consider
matrices $C_m(A)$ and $C_n(B)$.  Then, by the definition of the
constraint \eqref{eq:constraint2}, we have the Jacobian matrix
$J_{\bm{q}}(\bm{x})$ (with the original notation of variables for
$\bm{x}$ as in \eqref{eq:var0}) as
\begin{equation}
  \label{eq:jacobian-real}
  J_{\bm{q}}(\bm{x}) =
  \begin{pmatrix}
    \bm{0} & \bm{0} & 2\cdot{}^t\bm{v}\\
    C_m(A) & C_n(B) & N_{d-1}(\tilde{F},\tilde{G})
  \end{pmatrix}
  ,
\end{equation}
with $N_j(\tilde{F},\tilde{G})$ as in \eqref{eq:subresmat} and $\bm{v}$
as in \eqref{eq:abcond2-v}, respectively. Note that the matrix
$J_{\bm{q}}(\bm{x})$ has $m+n-d+2$ rows and $2(m+n-d+2)$ columns.

\subsubsection{The Complex Coefficient Case}

For co-factors $A(x)$ and $B(x)$ as in \eqref{eq:fgab-complex},
consider matrices $C_m(A)$ and $C_n(B)$ and express them as the sum of
matrices consisting of the real and the imaginary parts of whose
elements, respectively, as
\begin{equation*}
  \begin{split}
    C_m(A) &=
    \begin{pmatrix}
      a_{n-d,1} & & \\
      \vdots & \ddots & \\
      a_{0,1} & & a_{n-d,1} \\
      & \ddots & \vdots \\
      & & a_{0,1}
    \end{pmatrix}
    +\imu
    \begin{pmatrix}
      a_{n-d,2} & & \\
      \vdots & \ddots & \\
      a_{0,2} & & a_{n-d,2} \\
      & \ddots & \vdots \\
      & & a_{0,2}
    \end{pmatrix}
    \\
    &
    = C_m(A)_1 + \imu C_m(A)_2,
    \\
    C_n(B) &=
    \begin{pmatrix}
      b_{m-d,1} & & \\
      \vdots & \ddots & \\
      b_{0,1} & & b_{m-d,1} \\
      & \ddots & \vdots \\
      & & b_{0,1}
    \end{pmatrix}
    +\imu
    \begin{pmatrix}
      b_{m-d,2} & & \\
      \vdots & \ddots & \\
      b_{0,2} & & b_{m-d,2} \\
      & \ddots & \vdots \\
      & & b_{0,2}
    \end{pmatrix}
    \\
    &
    = C_n(B)_1 + \imu C_n(B)_2,
  \end{split}
\end{equation*}
respectively, and define
\begin{equation}
  \label{eq:a1a2}
  \begin{split}
    A_1 &= [C_m(A)_1\; C_n(B)_1]
    =
    \setlength{\arraycolsep}{1pt}
    \begin{pmatrix}
      a_{n-d,1} & & & b_{m-d,1} & & \\
      \vdots & \ddots & & \vdots & \ddots & \\
      a_{0,1} & & a_{n-d,1} & b_{0,1} & & b_{m-d,1} \\
      & \ddots & \vdots & & \ddots & \vdots \\
      & & a_{0,1} & & & b_{0,1}
    \end{pmatrix}
    ,
    \\
    A_2 &= [C_m(A)_2\; C_n(B)_2]
    =
    \setlength{\arraycolsep}{1pt}
    \begin{pmatrix}
      a_{n-d,2} & & & b_{m-d,2} & & \\
      \vdots & \ddots & & \vdots & \ddots & \\
      a_{0,2} & & a_{n-d,2} & b_{0,2} & & b_{m-d,2} \\
      & \ddots & \vdots & & \ddots & \vdots \\
      & & a_{0,2} & & & b_{0,2}
    \end{pmatrix}
    .
  \end{split}
\end{equation}
(Note that $A_1$ and $A_2$ are matrices of the real numbers of
$m+n-d+1$ rows and $m+n+2$ columns.)  Then, by the definition of the
constraint \eqref{eq:constraint-complex2}, we have the Jacobian matrix
$J_{\bm{q}}(\bm{x})$ (with the original notation of variables for
$\bm{x}$ as in \eqref{eq:var0-complex}) as
\begin{equation}
  \label{eq:jacobian-complex}
  J_{\bm{q}}(\bm{x}) =
  \begin{pmatrix}
    \bm{0} & \bm{0} & 2\cdot{}^t\bm{v}_1 & 2\cdot{}^t\bm{v}_2 \\
    A_1 & -A_2 & N_1 & -N_2 \\
    A_2 & A_1 & N_2 & N_1
  \end{pmatrix}
  ,
\end{equation}
with $A_1$ and $A_2$ as in \eqref{eq:a1a2} and $N_1$, $N_2$,
$\bm{v}_1$ and $\bm{v}_2$ as in \eqref{eq:n1n2v1v2}, respectively.

\subsection{Stability of the Algorithm}
\label{sec:jacobianmatrank}

In this paper, we treat the notion of ``stability'' of the algorithm
as to keep that the Jacobian $J_{\bm{q}}(\bm{x})$ in
Algorithms~\ref{alg:gp} and \ref{alg:gpnewton} has full rank, whereas
we usually discuss \textit{stability} as a notion in backward and/or
forward error analysis of numerical algorithms \cite{hig2002}.

In executing Algorithm~\ref{alg:gp} or \ref{alg:gpnewton}, we need the
algorithm to be \textit{stable} in the sense that we need to keep that
$J_{\bm{q}}(\bm{x})$ has full rank: otherwise, we cannot correctly
calculate $(J_{\bm{q}}(\bm{x}))^+$ (in Algorithm~\ref{alg:gp}) or the
matrix in \eqref{eq:modnewton} becomes singular (in
Algorithm~\ref{alg:gpnewton}) thus we are unable to decide proper
search direction.  For this requirement, we have the following
observations.

\begin{prop}
  \label{prop:fullrank}
  Let $\bm{x}^*\in V_{\bm{q}}$ be any feasible point satisfying Eq.\
  \eqref{eq:constraint2}.  Then, if the corresponding polynomials do
  not have a GCD whose degree exceeds $d$, then $J_{\bm{q}}(\bm{x}^*)$
  has full rank.
\end{prop}
\begin{proof}
  We prove the proposition in the real and the complex coefficient
  cases separately.

  \subsubsection{The Real Coefficient Case}

  Let $\bm{x}^*=(\tilde{f}_m,\ldots,\tilde{f}_0$,
  $\tilde{g}_n,\ldots,\tilde{g}_0$, $a_{n-d}\ldots,a_0$,
  $b_{m-d},\ldots,b_0)$ with its polynomial representation expressed
  as in \eqref{eq:fgab} (note that this assumption permits the
  polynomials $\tilde{F}(x)$ and $\tilde{G}(x)$ to be relatively prime
  in general).  To verify our claim, we show that we have
  $\rank(J_{\bm{q}}(\bm{x}^*))=m+n-d+2$ with $J_{\bm{q}}(\bm{x}^*)$ as
  in \eqref{eq:jacobian-real}.  Let us express $J_{\bm{q}}(\bm{x}^*)=
  \begin{pmatrix}
    J_\mathrm{L} \mid J_\mathrm{R}
  \end{pmatrix}
  $, where $J_\mathrm{L}$ and $J_\mathrm{R}$ are column blocks
  expressed as
  \[
  J_\mathrm{L} =
  \begin{pmatrix}
    \bm{0} & \bm{0} \\
    C_m(A) & C_n(B)
  \end{pmatrix}
  ,
  \quad
  J_\mathrm{R} =
  \begin{pmatrix}
    2\cdot\bm{v} \\
    N_{d-1}(\tilde{F},\tilde{G})
  \end{pmatrix}
  ,
  \]
  respectively.  Then, we have the following lemma.
  \begin{lem}
    \label{lem:welldefined}
    We have $\rank(J_\mathrm{L})=m+n-d+1$.
  \end{lem}
  \begin{proof}
    Let us express $J_\mathrm{L}=
    \begin{pmatrix}
      J_\mathrm{LL} \mid J_\mathrm{LR}
    \end{pmatrix}
    $, where
    \[
    J_\mathrm{LL} = 
    \begin{pmatrix}
      \bm{0} \\
      C_m(A)
    \end{pmatrix}
    ,
    \quad
    J_\mathrm{LR} = 
    \begin{pmatrix}
      \bm{0} \\
      C_n(B)
    \end{pmatrix}
    ,
    \]
    and let $\bar{J}_\mathrm{L}$ be a submatrix of $J_\mathrm{L}$ by
    taking the right $m-d$ columns of $J_\mathrm{LL}$ and the right
    $n-d$ columns of $J_\mathrm{LR}$.
    Then,
    we see that the bottom $m+n-2d$ rows of $\bar{J}_\mathrm{L}$ is
    equal to $N(A,B)$, the Sylvester matrix of $A(x)$ and $B(x)$.  By
    the assumption, polynomials $A(x)$ and $B(x)$ are relatively
    prime, and there exist no nonzero elements in $\bar{J}_\mathrm{L}$
    except for the bottom $m+n-2d$ rows, we have
    $\rank(\bar{J}_\mathrm{L})=m+n-2d$.

    By the above structure of $\bar{J}_\mathrm{L}$ and the lower
    triangular structure of $J_\mathrm{LL}$ and $J_\mathrm{LR}$, we can
    take the left $d+1$ columns of $J_\mathrm{LL}$ or $J_\mathrm{LR}$
    satisfying linear independence along with the $m+n-2d$ columns in
    $\bar{J}_\mathrm{L}$.  Therefore, these $m+n-d+1$ columns generate
    a $(m+n-d+1)$-dimensional subspace in $\R^{m+n-d+2}$ satisfying
    \begin{equation}
      \label{eq:proposition-subspace}
      \{
      {}^t(x_1,\ldots,x_{m+n-d+2})\in\R^{m+n-d+2}\mid x_1=0
      \},
    \end{equation}
    and we see that none of the columns in $J_\mathrm{L}$ have nonzero
    element in the top coordinate.  This proves the lemma.
  \end{proof}

  \textit{Proof of Proposition~\ref{prop:fullrank} (in the real
    coefficient case, continued).}
  By the assumptions, we have at least one column vector in
  $J_\mathrm{R}$ with nonzero coordinate on the top row.  By adding
  such a column vector to the basis of the subspace
  \eqref{eq:proposition-subspace} that are generated as in
  Lemma~\ref{lem:welldefined}, we have a basis of $\R^{m+n-d+2}$.
  This implies $\rank(J_{\bm{q}}(\bm{x}))=m+n-d+2$, which proves the
  proposition in the real coefficient case.

  \subsubsection{The Complex Coefficient Case}

  Let $\bm{x}^*= ( \tilde{f}_{m,1},\ldots,\tilde{f}_{0,1},
  \tilde{g}_{n,1},\ldots,\tilde{g}_{0,1},
  \tilde{f}_{m,2},\ldots,\tilde{f}_{0,2},
  \tilde{g}_{n,2},\ldots,\tilde{g}_{0,2}, a_{n-d,1},\ldots,a_{0,1},$
  \linebreak
  $b_{m-d,1},\ldots,b_{0,1}, a_{n-d,2},\ldots,a_{0,2},
  b_{m-d,2},\ldots,b_{0,2} ) $ with its polynomial representation
  expressed as in \eqref{eq:fgab-complex} (note that this assumption
  permits the polynomials $\tilde{F}(x)$ and $\tilde{G}(x)$ to be
  relatively prime in general).  To verify our claim, we show that we
  have $\rank(J_{\bm{q}}(\bm{x}^*))=2(m+n-d+1)+1$ as in
  \eqref{eq:rankj}, with $J_{\bm{q}}(\bm{x}^*)$ as in
  \eqref{eq:jacobian-complex}.  Let us express
  $J_{\bm{q}}(\bm{x}^*)=
  \begin{pmatrix}
    J_\mathrm{L} \mid J_\mathrm{R}
  \end{pmatrix}
  $, where $J_\mathrm{L}$ and $J_\mathrm{R}$ are column blocks
  expressed as
  \[
  J_\mathrm{L} =
  \begin{pmatrix}
    \bm{0} & \bm{0} \\
    A_1 & -A_2 \\
    A_2 & A_1 
  \end{pmatrix}
  ,
  \quad
  J_\mathrm{R} =
  \begin{pmatrix}
    2\cdot{}^t\bm{v}_1 & 2\cdot{}^t\bm{v}_2 \\
    N_1 & -N_2 \\
    N_2 & N_1
  \end{pmatrix}
  ,
  \]
  respectively.  Then, we have the following lemma.
  \begin{lem}
    \label{lem:fullrank}
    We have $\rank(J_\mathrm{L})=2(m+n-d+1)$.
  \end{lem}
  \begin{proof}
    For $A_1=[C_m(A)_1\; C_n(B)_1]$, let $\overline{C_m(A)_1}$ be the
    right $m-d$ columns of $C_m(A)_1$ and $\overline{C_n(B)_1}$ be the
    right $n-d$ columns of $C_n(B)_1$.  Then, we see that the bottom
    $m+n-2d$ rows of the matrix $\bar{C}=[\overline{C_m(A)_1}\;
    \overline{C_n(B)_1}]$ is equal to the matrix consisting of the
    real part of the elements of $N(A,B)$, the Sylvester matrix of
    $A(x)$ and $B(x)$.  By the assumption, polynomials $A(x)$ and
    $B(x)$ are relatively prime, and there exist no nonzero elements
    in $\bar{C}$ except for the bottom $m+n-2d$ rows, thus we have
    $\rank(\bar{C})=m+n-2d$.

    By the structure of $\bar{C}$ and the lower triangular structure
    of $C_m(A)_1$ and $C_n(B)_1$, we can take the left $d+1$ columns
    of $C_m(A)_1$ or $C_n(B)_1$ satisfying linear independence along
    with $\bar{C}$, which implies that there exist a nonsingular
    square matrix $T$ of order $m+n+2$ satisfying
    \begin{equation}
      \label{eq:tr1}
      A_1 T=R,
    \end{equation}
    where $R$ is a lower triangular matrix, thus we have
    $\rank(A_1)=\rank(R)=m+n-d+1$.

    Furthermore, by using $T$ and $R$ in \eqref{eq:tr1}, we have
    \begin{equation}
      \label{eq:tr2}
      \begin{pmatrix}
        \bm{0} & \bm{0} \\
        A_1 & -A_2 \\
        A_2 & A_1
      \end{pmatrix}
      \begin{pmatrix}
        T & \bm{0} \\
        \bm{0} & T
      \end{pmatrix}
      =
      \begin{pmatrix}
        \bm{0} & \bm{0} \\
        R & -A_2T \\
        A_2T & R
      \end{pmatrix}
      ,
    \end{equation}
    followed by a suitable transformation on columns on the matrix in
    the right-hand-side of \eqref{eq:tr2}, we can make $A_2T$ to zero
    matrix, which implies that
    \[
    \rank(J_\mathrm{L})=\rank
    \left(
      \begin{pmatrix}
        \bm{0} & \bm{0} \\
        R & -A_2T \\
        A_2T & R
      \end{pmatrix}
    \right)
    =2\cdot\rank(R)=2(m+n-d+1).
    \]
    This proves the lemma.
  \end{proof}

  \textit{Proof of Proposition~\ref{prop:fullrank} (in the complex
    coefficient case, continued).}  By the assumptions, we have at
  least one nonzero coordinate in the top row in $J_\mathrm{R}$, while
  we have no nonzero coordinate in the top row in $J_\mathrm{L}$, thus
  we have $\rank(J_{\bm{q}}(\bm{x}))=2(m+n-d+1)+1$, which proves the
  proposition in the complex coefficient case.
\end{proof}

\begin{rem}
  \label{rem:jacobian}
Proposition~\ref{prop:fullrank} says that, so long as the search
direction in the minimization problem satisfies that corresponding
polynomials have a GCD of degree not exceeding $d$, then
$J_{\bm{q}}(\bm{x})$ has full rank, thus we can safely calculate the
next search direction for approximate GCD.  On the other hand, it is still
not clear when $J_{\bm{q}}(\bm{x})$ becomes singular in our
minimization problem.  Although our experiments have shown that the
iteration converges for any $d$ satisfying $0<d\le n$ in many
examples, its theoretical property deserves further investigation (see
also concluding remarks (Section~\ref{sec:remark})).
\end{rem}

\subsection{Setting the Initial Values}
\label{sec:init}

At the beginning of iterations, we give the initial value $\bm{x}_0$
by using the singular value decomposition (SVD) (\cite{dem1997}), as
follows.

\subsubsection{The Real Coefficient Case}

In the case of the real coefficients, we calculate the SVD of the
$(d-1)$-th subresultant matrix $N_{d-1}(F,G):
\R^{m+n-2d+2}\rightarrow\R^{m+n-d+1}$ (see \eqref{eq:subresmat}).  Let
$N_{d-1}(F,G) = U\,\Sigma\,{}^tV$ be the SVD of $N_{d-1}(F,G)$, where
\begin{equation}
  \label{eq:svd-nd-1}
  \begin{array}{c}
    N_{d-1}(F,G) =  U\,\Sigma\,{}^tV,
    \quad
    U = (\bm{u}_1,\ldots,\bm{u}_{m+n-2d+2}),
    \\
    \noalign{\vskip2pt}
    \Sigma = \diag(\sigma_1,\ldots,\sigma_{m+n-2d+2}),
    \quad
    V = (\bm{v}_1,\ldots,\bm{v}_{m+n-2d+2}),
  \end{array}
\end{equation}
with $\bm{u}_j\in\R^{m+n-d+1}$, $\bm{v}_j\in\R^{m+n-2d+2}$, and
$\Sigma=\diag(\sigma_1,\ldots,$ $\sigma_{m+n-2d+2})$ denotes the diagonal
matrix whose the $j$-th diagonal element is $\sigma_j$.  Note that $U$
and $V$ are orthogonal matrices.  Then, by a property of the SVD
(\cite[Theorem~3.3]{dem1997}), the smallest singular value
$\sigma_{m+n-2d+2}$ gives the minimum distance of the image of the unit
sphere
$\textrm{S}^{(m+n-2d+2)-1}$, given as
\[
\textrm{S}^{(m+n-2d+2)-1}=
\{
\bm{x}\in\R^{m+n-2d+2} \mid \|\bm{x}\|_2=1
\}
,
\]
by $N_{d-1}$, represented as
\[
N_{d-1}\cdot\textrm{S}^{(m+n-2d+2)-1}=
\{
N_{d-1}\bm{x}\mid \bm{x}\in\R^{m+n-2d+2}, \|\bm{x}\|_2=1
\},
\]
from the origin, along with $\sigma_{m+n-2d+2}\bm{u}_{m+n-2d+2}$ as its
coordinates.  By \eqref{eq:svd-nd-1}, we have
\[
N_{d-1}\cdot \bm{v}_{m+n-2d+2} = \sigma_{m+n-2d+2}\bm{u}_{m+n-2d+2},
\]
thus  $\bm{v}_{m+n-2d+2}$ represents the coefficients
of $A(x)$ and $B(x)$: let
\[
\begin{split}
  \bm{v}_{m+n-2d+2} &=
  {}^t(\bar{a}_{n-d},\ldots,\bar{a}_0,\bar{b}_{m-d},\ldots,\bar{b}_0),
  \\
  \bar{A}(x) &= \bar{a}_{n-d}x^{n-d}+\cdots+\bar{a}_0x^0, \\
  \bar{B}(x) &= \bar{b}_{m-d}x^{m-d}+\cdots+\bar{b}_0x^0.
\end{split}
\]
Then, $\bar{A}(x)$ and $\bar{B}(x)$ give the least norm of $AF+BG$
satisfying $\|A\|_2^2+\|B\|_2^2=1$ by putting $A(x)=\bar{A}(x)$ and
$B(x)=\bar{B}(x)$. 

Therefore, we admit the coefficients of $F$, $G$, $\bar{A}$ and
$\bar{B}$ as the initial values of the iterations as
\begin{equation}
  \label{eq:appgcd-init}
  \bm{x}_0 = 
  (f_m,\ldots,f_0,g_n,\ldots,g_0,
  \bar{a}_{n-d},\ldots,\bar{a}_0,\bar{b}_{m-d},\ldots,\bar{b}_0).
\end{equation}

\subsubsection{The Complex Coefficient Case}

In the complex case, we calculate the
SVD of 
$
N=
\begin{pmatrix}
  N_1 & -N_2 \\
  N_2 & N_1
\end{pmatrix}
$
in \eqref{eq:abcond2-complex-3} as
\begin{equation}
  \label{eq:svd-nd-1-complex}
  \begin{array}{c}
    N =  U\,\Sigma\,{}^tV,
    \quad
    U = (\bm{u}_1,\ldots,\bm{u}_{2(m+n-2d+2)}),
    \\
    \noalign{\vskip2pt}
    \Sigma = \diag(\sigma_1,\ldots,\sigma_{2(m+n-2d+2)}),
    \quad
    V = (\bm{v}_1,\ldots,\bm{v}_{2(m+n-2d+2)}),
  \end{array}
\end{equation}
where $\bm{u}_j\in\R^{2(m+n-d+1)}$, $\bm{v}_j\in\R^{2(m+n-2d+2)}$, and
$U$ and $V$ are orthogonal matrices.  Then, as in the case of the real
coefficients, the smallest singular value $\sigma_{2(m+n-2d+2)}$ gives
the minimum distance of the image of the unit sphere\linebreak
$\textrm{S}^{2(m+n-2d+2)-1}$, given as
\[
\textrm{S}^{2(m+n-2d+2)-1}=
\{
\bm{x}\in\R^{2(m+n-2d+2)} \mid \|\bm{x}\|_2=1
\}
,
\]
by $N$, represented as
\[
N\cdot\textrm{S}^{2(m+n-2d+2)-1}=
\{
N\bm{x}\mid \bm{x}\in\R^{2(m+n-2d+2)}, \|\bm{x}\|_2=1
\},
\]
from the origin, along with $\sigma_{2(m+n-2d+2)}\bm{u}_{2(m+n-2d+2)}$
as its coordinates.  By \eqref{eq:svd-nd-1-complex}, we have
\[
N\cdot \bm{v}_{2(m+n-2d+2)} = \sigma_{2(m+n-2d+2)}\bm{u}_{2(m+n-2d+2)},
\]
thus  $\bm{v}_{2(m+n-2d+2)}$ represents the coefficients
of $A(x)$ and $B(x)$: let
\[
\begin{split}
  \bm{v}_{2(m+n-2d+2)} &=
  {}^t(
  \bar{a}_{n-d,1},\ldots,\bar{a}_{0,1},
  \bar{b}_{m-d,1},\ldots,\bar{b}_{0,1},
  \bar{a}_{n-d,2},\ldots,\bar{a}_{0,2},
  \bar{b}_{m-d,2},\ldots,\bar{b}_{0,2}
  ),
  \\
  \bar{A}(x) &= (\bar{a}_{n-d,1}+\bar{a}_{n-d,2}\imu)x^{n-d}
  +\cdots+
  (\bar{a}_{0,1}+\bar{a}_{0,2}\imu) x^0, \\
  \bar{B}(x) &= (\bar{b}_{m-d,1}+\bar{b}_{m-d,2}\imu)x^{m-d}
  +\cdots+
  (\bar{b}_{0,1}+\bar{b}_{0,2}\imu) x^0.
\end{split}
\]
Then, $\bar{A}(x)$ and $\bar{B}(x)$ give the least norm of $AF+BG$
satisfying $\|A\|_2^2+\|B\|_2^2=1$ by putting $A(x)=\bar{A}(x)$ and
$B(x)=\bar{B}(x)$ in \eqref{eq:fgab-complex}. 

Therefore, we admit the coefficients of $F$, $G$, $\bar{A}$ and
$\bar{B}$ as the initial values of the iterations as
\begin{multline}
  \label{eq:appgcd-init-complex}
  \bm{x}_0 = 
  (
  f_{m,1},\ldots,f_{0,1},
  g_{n,1},\ldots,g_{0,1},
  f_{m,2},\ldots,f_{0,2},
  g_{n,2},\ldots,g_{0,2},
  \\
  \bar{a}_{n-d,1},\ldots,\bar{a}_{0,1},
  \bar{b}_{m-d,1},\ldots,\bar{b}_{0,1},
  \bar{a}_{n-d,2},\ldots,\bar{a}_{0,2},
  \bar{b}_{m-d,2},\ldots,\bar{b}_{0,2}
  ).
\end{multline}

\subsection{Regarding the Minimization Problem as the Minimum Distance
  (Least Squares) Problem} 
\label{sec:mindistance}

Since we have the object function $f$ as in \eqref{eq:objective1} or
\eqref{eq:objective-complex2} in the case of the real or the complex
coefficients, respectively, we have $\nabla
f(\bm{x})=2\bm{v}_{\text{R}}$, where
\begin{equation}
  \label{eq:vr}
  \bm{v}_{\text{R}}=
  {}^t(
  x_1-f_m,\ldots,x_{m+1}-f_0,
  x_{m+2}-g_n,\ldots,x_{m+n+2}-g_0,
  0,\ldots,0),
\end{equation}
in the case of the real coefficients, or $\nabla
f(\bm{x})=2\bm{v}_{\text{C}}$, where
\begin{multline}
  \label{eq:vc}
  \bm{v}_{\text{C}}=
  {}^t(
  x_1-f_{m,1},\ldots,x_{m+1}-f_{0,1},
  x_{m+2}-g_{n,1},\ldots,x_{m+n+2}-g_{0,1},
  \\
  x_{m+n+3}-f_{m,2},\ldots,x_{2m+n+3}-f_{0,2},
  \\
  x_{2m+n+4}-g_{n,2},\ldots,x_{2(m+n+2)}-g_{0,2},
  0,\ldots,0),
\end{multline}
in the case of the complex coefficients, respectively.  However, we
can regard our problem as finding a point $\bm{x}\in V_{\bm{q}}$ which
has the minimum distance to the initial point $\bm{x}_0$ with respect
to the $(x_1,\ldots,x_{m+n+2})$-coordinates in the case of the real
coefficients or the $(x_1,\ldots,x_{2(m+n+2)})$-coordinates in the
case of the complex coefficients, respectively, which correspond to the
coefficients in $F(x)$ and $G(x)$.  Therefore, in the gradient
projection method at $\bm{x}\in V_{\bm{q}}$, the projection of
$-\nabla f(\bm{x})$ in \eqref{eq:projection} should be the projection
of
$\bm{v}_{\text{R}}$
in the case of the real coefficients, or
$\bm{v}_{\text{C}}$
in the case of the complex coefficients, respectively, onto
$T_{\bm{x}}$, where $\bm{v}_{\text{R}}$ and $\bm{v}_{\text{C}}$ are as
in \eqref{eq:vr} and \eqref{eq:vc}, respectively. These changes are
equivalent to changing the objective function as
$\bar{f}(\bm{x})=\frac{1}{2}f(\bm{x})$ then solving the minimization
problem of $\bar{f}(\bm{x})$, subject to $\bm{q}(\bm{x})=\bm{0}$.

\subsection{Calculating the Actual GCD and Correcting the Deformed
  Polynomials}
\label{sec:correct}

After successful end of the iterations in Algorithms~\ref{alg:gp} or
\ref{alg:gpnewton}, we obtain the coefficients of $\tilde{F}(x)$,
$\tilde{G}(x)$, $A(x)$ and $B(x)$ satisfying \eqref{eq:abcond} with
$A(x)$ and $B(x)$ are relatively prime.  Then, we need to compute the
actual GCD $H(x)$ of $\tilde{F}(x)$ and $\tilde{G}(x)$.  Although $H$
can be calculated as the quotient of $\tilde{F}$ divided by $B$ or
$\tilde{G}$ divided by $A$, naive polynomial division may cause
numerical errors in the coefficient.  Thus, we calculate the
coefficients of $H$ by the so-called least squares division
(\cite{zen2011}), followed by correcting the coefficients in
$\tilde{F}$ and $\tilde{G}$ by using the calculated $H$, as follows.

\subsubsection{Calculating Candidates for the GCD in the Real
  Coefficient Case} 

For polynomials $\tilde{F}$, $\tilde{G}$, $A$ and $B$
represented as in \eqref{eq:fgab} and $H$ represented as
\[
H(x)=h_dx^d+\cdots+h_0x^0,
\]
solve the equations $HB=\tilde{F}$ and $HA=\tilde{G}$ with respect to
$H$ as solving the least squares problems of linear systems
\begin{align}
  \label{eq:hsystem1}
  C_d(A)\,{}^t(h_d,\ldots,h_0) &=
  {}^t(\tilde{g}_n,\ldots,\tilde{g}_0), \\
  \label{eq:hsystem2}
  C_d(B)\,{}^t(h_d,\ldots,h_0) &=
  {}^t(\tilde{f}_m,\ldots,\tilde{f}_0),
\end{align}
respectively.  Let $H_1(x),H_2(x)\in\R[x]$ be the candidates for the
GCD whose coefficients are calculated as the least squares solutions
of \eqref{eq:hsystem1} and \eqref{eq:hsystem2}, respectively.

\subsubsection{Calculating Candidates for the GCD in the Complex
  Coefficient Case}

For polynomials $\tilde{F}$, $\tilde{G}$, $A$ and $B$
represented as in \eqref{eq:fgab-complex} and $H$ represented as
\[
H(x)=(h_{d,1}+h_{d,2}\imu)x^d+\cdots+(h_{0,1}+h_{0,2}\imu)x^0,
\]
solve the equations $HB=\tilde{F}$ and $HA=\tilde{G}$ with respect to
$H$ as solving the least squares problems of linear systems
\begin{align}
  \label{eq:hsystem1-complex}
  C_d(A)\,{}^t(h_{d,1}+h_{d,2}\imu,\ldots,h_{0,1}+h_{0,2}\imu) &=
  {}^t(\tilde{g}_{n,1}+\tilde{g}_{n,2}\imu,\ldots,
  \tilde{g}_{0,1}+\tilde{g}_{0,2}\imu), \\
  \label{eq:hsystem2-complex}
  C_d(B)\,{}^t(h_{d,1}+h_{d,2}\imu,\ldots,h_{0,1}+h_{0,2}\imu) &=
  {}^t(\tilde{f}_{m,1}+\tilde{f}_{m,2}\imu,\ldots,
  \tilde{f}_{0,1}+\tilde{f}_{0,2}\imu),
\end{align}
respectively.  Then, we transfer the linear systems
\eqref{eq:hsystem1-complex} and \eqref{eq:hsystem2-complex}, as
follows.  For \eqref{eq:hsystem2-complex}, let us express the matrices
and vectors as the sum of the real and the imaginary part of which,
respectively, as
\begin{gather*}
  C_d(B) = B_1+\imu B_2,
  \\
  {}^t(h_{d,1}+h_{d,2}\imu,\ldots,h_{0,1}+h_{0,2}\imu) =
  \bm{h}_1+\imu\bm{h}_2,
  \\
  {}^t(\tilde{f}_{m,1}+\tilde{f}_{m,2}\imu,\ldots,
  \tilde{f}_{0,1}+\tilde{f}_{0,2}\imu) =
  \bm{f}_1+\imu\bm{f}_2.
\end{gather*}
Then, \eqref{eq:hsystem2} is expressed as
\begin{equation}
  \label{eq:hsystem2-2}
  (B_1+\imu B_2)(\bm{h}_1+\imu\bm{h}_2)=(\bm{f}_1+\imu\bm{f}_2).
\end{equation}
By equating the real and the imaginary parts in Eq.\
\eqref{eq:hsystem2-2}, respectively, we have
\[
(B_1\bm{h}_1-B_2\bm{h}_2)=\bm{f}_1,\quad
(B_1\bm{h}_2+B_2\bm{h}_1)=\bm{f}_2,
\]
or
\begin{equation}
  \label{eq:hsystem2-3}
  \begin{pmatrix}
    B_1 & -B_2 \\
    B_2 & B_1
  \end{pmatrix}
  \begin{pmatrix}
    \bm{h}_1 \\ \bm{h}_2
  \end{pmatrix}
  =
  \begin{pmatrix}
    \bm{f}_1 \\ \bm{f}_2
  \end{pmatrix}
  .
\end{equation}
Thus, we can calculate the coefficients of $H(x)$ by solving the real
least squares problem \eqref{eq:hsystem2-3}.  We can solve
\eqref{eq:hsystem1-complex} similarly.  Let $H_1(x),H_2(x)\in\C[x]$ be
the candidates for the GCD whose coefficients are calculated as the
least squares solutions of \eqref{eq:hsystem1-complex} and
\eqref{eq:hsystem2-complex}, respectively.

\subsubsection{Choosing the GCD and Calculating the Deformed Polynomials}

Let $H_1(x),H_2(x)\in\C[x]$ be the candidates for the GCD calculated
as in the above.  Then, for $i=1,2$, calculate the norms of the
residues as
\[
r_i = \|\tilde{F}-H_iB\|_2^2+\|\tilde{G}-H_iA\|_2^2,
\]
respectively, and set the GCD $H(x)$ be $H_i(x)$ giving the minimum
value of $r_i$.

Finally, for the chosen $H(x)$, correct the coefficients of
$\tilde{F}(x)$ and $\tilde{G}(x)$ as
\[
\tilde{F}(x)=H(x)\cdot B(x), \quad
\tilde{G}(x)=H(x)\cdot A(x),
\]
respectively.

\subsection{The Algorithm}

Summarizing the above, the algorithm for calculating approximate GCD
becomes as follows.

\begin{alg}[GPGCD: Approximate GCD by the
    Gradient-Projection Method] 
  \label{alg:gpgcd}
  \mbox{\\}
  \begin{itemize}
  \item Inputs:
    \begin{itemize}
    \item $F(x),G(x)\in\R[x]$ or $\C[x]$ with $\deg(F)\ge\deg(G)>0$,
    \item $d\in\N$: the degree of approximate GCD with $d\le \deg(G)$,
    \item $\varepsilon>0$: a threshold for terminating iteration in
      the gradient-projection method,
    \item $u\in\N$: an upper bound for the number of iterations
      permitted in the gradient-projection method.
    \end{itemize}
  \item Outputs: $\tilde{F}(x),\tilde{G}(x),H(x)\in\R[x]$ or $\C[x]$
    such that $\tilde{F}$ and $\tilde{G}$ are deformations of $F$ and
    $G$, respectively, whose GCD is equal to $H$ with $\deg(H)=d$.
  \end{itemize}
  \begin{description}
  \item[Step 1] \textbf{[Setting the initial values]} As the
    discussions in Section~\ref{sec:init}, set the initial values
    $\bm{x}_0$ as in \eqref{eq:appgcd-init} in the case of the real
    coefficients, or \eqref{eq:appgcd-init-complex} in the case of the
    complex coefficients, respectively.
  \item[Step 2] \textbf{[Iteration]} As the discussions in
    Section~\ref{sec:mindistance}, solve the minimization problem of
    $\bar{f}(\bm{x})=\frac{1}{2}f(\bm{x})$, subject to
    $\bm{q}(\bm{x})=\bm{0}$, with $f(\bm{x})$ and $\bm{q}(\bm{x})$ as
    in \eqref{eq:objective1} and \eqref{eq:constraint2} in the case of
    the real coefficients, or in \eqref{eq:objective-complex2} and
    \eqref{eq:constraint-complex2} in the case of the complex
    coefficients, respectively.  Apply Algorithm~\ref{alg:gp} or
    \ref{alg:gpnewton} for the minimization: repeat iterations until
    the search direction $\bm{d}_k$ (as in \eqref{eq:projection} in
    the gradient-projection method or in \eqref{eq:modnewtonsol} in a
    modified Newton method, respectively) satisfies
    $\|\bm{d}_k\|_2<\varepsilon$, or the number of iteration reaches
    its upper bound $u$.
  \item[Step 3] \textbf{[Construction of $\tilde{F}$, $\tilde{G}$ and
      $H$]} As the discussions in Section~\ref{sec:correct},
    construct the GCD $H(x)$ and correct the coefficients of
    $\tilde{F}(x)$ and $\tilde{G}(x)$.  Then, return $\tilde{F}(x)$,
    $\tilde{G}(x)$ and $H(x)$.  If Step 2 did not end with the number
    of iterations less than $u$, report it to the user.
  \end{description}
\end{alg}

\subsection{Preserving Monicity}

While Algorithm~\ref{alg:gpgcd} permits changing the leading
coefficients for calculating $\tilde{F}(x)$ and $\tilde{G}(x)$, we can
also give an algorithm restricting inputs $F(x)$ and $G(x)$ and
outputs $\tilde{F}(x)$ and $\tilde{G}(x)$ to be monic as follows.

\subsubsection{The Real Coefficient Case}

Let $\tilde{F}(x)$ and $\tilde{G}(x)$ be represented as in
\eqref{eq:fgab} with $\tilde{f}_m=\tilde{g}_n=1$, then, by 
Eq.\ \eqref{eq:abcond2},
we have $b_{m-d}=-a_{n-d}$.  Thus, we eliminate the variables
$\tilde{f}_m$, $\tilde{g}_n$ and $b_{m-d}$, which cause the following
changes. 

\paragraph{Changes on the Subresultant Matrix}
\label{sec:change-subresmat}

By eliminating the variables as in the above, we see that Eq.\
\eqref{eq:abcond2} is equivalent to
\[
N'_{d-1}(\tilde{F},\tilde{G})
\cdot{}^t(a_{n-d},\ldots,a_0,b_{m-d-1},\ldots,b_0)
=\bm{0}, 
\]
where $N'_{d-1}(\tilde{F},\tilde{G})$ is defined as
\begin{equation*}
  N'_{d-1}(\tilde{F},\tilde{G})=
  \setlength{\arraycolsep}{1pt}
  \begin{pmatrix}
    \tilde{f}_{m-1}-\tilde{g}_{n-1} & 1 & & & 1 & &  \\
    \vdots & \tilde{f}_{m-1} & \ddots &  & \tilde{g}_{n-1} & \ddots & \\ 
    \tilde{f}_0-\tilde{g}_{n-m} & \vdots & \ddots & 1 &
    \vdots & \ddots & 1 \\
    & \tilde{f}_0 & & \tilde{f}_{m-1} & \tilde{g}_0 & &
    \tilde{g}_{n-1} \\
    & & \ddots & \vdots & & \ddots & \vdots \\
    & & & \tilde{f}_0 & & & \tilde{g}_0
  \end{pmatrix}
  ,
\end{equation*}
with (in the first column) $\tilde{g}_j=0$ for $j<0$,
by subtracting the first column by the $(n-d+1)$-th column, then
deleting the first row and the $(n-d+1)$-th column (corresponding to
the $b_{m-d}$ term) in $N_{d-1}(\tilde{F},\tilde{G})$.

\paragraph{Changes on the Settings in the Minimization Problem}
\label{sec:minimize-monic}

In solving the minimization problem, we substitute the variables
\begin{equation*}
  (\tilde{f}_{m-1},\ldots,\tilde{f}_0,
  \tilde{g}_{n-1},\ldots,\tilde{g}_0,
  a_{n-d},\ldots,a_0, b_{m-d-1},\ldots,b_0)
\end{equation*}
as $\bm{x}=(x_1,\ldots,x_{2(m+n-d)+1})$, instead of \eqref{eq:var0}.
As a consequence, in contrast to \eqref{eq:objective1}, the objective
function $f(\bm{x})$ becomes as
\begin{multline}
  \label{eq:objective1-monic}
  f(\bm{x})=(x_1-f_{m-1})^2+\cdots+(x_m-f_0)^2
  \\
  +(x_{m+1}-g_{n-1})^2+\cdots+(x_{m+n}-g_0)^2.
\end{multline}
Also, in contrast to \eqref{eq:constraint0} and
\eqref{eq:constraint1}, the constraints 
$\bm{q}(\bm{x})$ become as
\begin{equation}
  \label{eq:constraint0-monic}
  \begin{split}
    q_0 &=
    2a_{n-d}^2+a_{n-d-1}^2\cdots+a_0^2+b_{m-d-1}^2+\cdots+b_0^2-1=0,
    \\
    q_1 &= (\tilde{f}_{m-1}-\tilde{g}_{n-1}) a_{n-d} + a_{n-d-1} +
    b_{m-d-1} = 0,
    \\
    &\quad\vdots
    \\
    q_{m+n-d} &= \tilde{f}_0 a_0 + \tilde{g}_0 b_0 = 0.
  \end{split}
\end{equation}

\paragraph{Changes on the Initial Values}
\label{sec:init-monic}

Let $N'_{d-1}=U\,\Sigma\,{}^tV$ be the SVD of $N'_{d-1}(F,G)$, with
\[
\begin{split}
  V &= (\bm{v}_1,\ldots,\bm{v}_{m+n-2d-1}),\\
  \bm{v}_{m+n-2d-1} &=
  {}^t(\bar{a}_{n-d},\ldots,\bar{a}_0,\bar{b}_{m-d-1},\ldots,\bar{b}_0).
\end{split}
\]
Then, in contrast to \eqref{eq:appgcd-init}, the initial values become
as 
\begin{equation}
  \label{eq:appgcd-init-monic}
  \bm{x}_0 = 
  (f_{m-1},\ldots,f_0,g_{n-1},\ldots,g_0,
  \bar{a}_{n-d},\ldots,\bar{a}_0,\bar{b}_{m-d-1},\ldots,\bar{b}_0).
\end{equation}

\paragraph{The Algorithm}

Summarizing discussions in the above, for preserving $\tilde{F}(x)$
and $\tilde{G}(x)$ to be monic, we modify Algorithm~\ref{alg:gpgcd} as
follows. 
\begin{alg}[GPGCD preserving monicity, with real coefficients]
  \label{alg:gpgcd-monic}
  Change Steps 1 and 2 in Algorithm~\ref{alg:gpgcd} as follows.
  \begin{description}
  \item[Step 1] \textbf{[Setting the initial values]} Set the initial values
    $\bm{x}_0$ as in \eqref{eq:appgcd-init-monic}.
  \item[Step 2] \textbf{[Iteration]} Solve the minimization problem of
    $\bar{f}(\bm{x})=\frac{1}{2}(\bm{x})$, subject to
    $\bm{q}(\bm{x})=\bm{0}$, with $f(\bm{x})$ and $\bm{q}(\bm{x})$
    defined as in \eqref{eq:objective1-monic} and
    \eqref{eq:constraint0-monic}, respectively, as Step 2 in
    Algorithm~\ref{alg:gpgcd}.
  \end{description}
\end{alg}

\subsubsection{The Complex Coefficient Case}

Let $\tilde{F}(x)$ and $\tilde{G}(x)$ be represented as in
\eqref{eq:fgab-complex} with $\tilde{f}_{m,1}=\tilde{g}_{n,1}=1$ and
$\tilde{f}_{m,2}=\tilde{g}_{n,2}=0$, then, by Eq.\ \eqref{eq:abcond2},
we have $b_{m-d,j}=-a_{n-d,j}$ for $j\in\{1,2\}$.  Thus, for
$j\in\{1,2\}$, we eliminate the variables $\tilde{f}_{m,j}$,
$\tilde{g}_{n,j}$ and $b_{m-d,j}$, which cause the following changes.

\paragraph{Changes on the Subresultant Matrix}

By eliminating the variables as in the above, we see that Eq.\
\eqref{eq:abcond2-complex} is equivalent to
\begin{multline}
  \label{eq:abcond2-complex-monic}
  \setlength{\arraycolsep}{1pt}
  \left(
  \begin{array}{cccc}
    (\tilde{f}_{m-1,1}-\tilde{g}_{n-1,1})+(\tilde{f}_{m-1,2}-\tilde{g}_{n-1,2})
    \imu
    & 1 & &  
    \\
    \vdots
    & \tilde{f}_{m-1,1}+\tilde{f}_{m-1,2}\imu & \ddots &  
    \\
    (\tilde{f}_{0,1}-\tilde{g}_{n-m,1})+(\tilde{f}_{0,2}-\tilde{g}_{n-m,2})
    \imu
    &\vdots      & \ddots &      1
    \\
    &\tilde{f}_{0,1}+\tilde{f}_{0,2}\imu & &
    \tilde{f}_{m-1,1}+\tilde{f}_{m-1,2}\imu 
    \\
    & & \ddots & \vdots 
    \\
    & & & \tilde{f}_{0,1}+\tilde{f}_{0,2}\imu
  \end{array}
  \right.
  \\
  \left.
  \begin{array}{ccc}
    1 & &  \\
    \tilde{g}_{n-1,1}+\tilde{g}_{n-1,2}\imu & \ddots &  \\
    \vdots      & \ddots &  1\\
    \tilde{g}_{0,1}+\tilde{g}_{0,2}\imu & &
    \tilde{g}_{n-1,1}+\tilde{g}_{n-1,2}\imu \\
    & \ddots      & \vdots \\
    & &
    \tilde{g}_{0,1}+\tilde{g}_{0,2}\imu 
  \end{array}
  \right)
  \begin{pmatrix}
    a_{n-d,1}+a_{n-d,2}\imu
    \\
    \vdots
    \\
    a_{0,1}+a_{0,2}\imu
    \\
    b_{m-d-1,1}+b_{m-d-1,2}\imu
    \\
    \vdots
    \\
    b_{0,1}+b_{0,2}\imu
  \end{pmatrix}
  \\
  =
  \bm{0}
  ,
\end{multline}
with (in the first column of the matrix in the left-hand-side)
$\tilde{g}_{i,j}=0$ for $i<0$ and $j\in\{1,2\}$, in which the matrix
in the left-hand-side is obtained by subtracting the first column by
the $(n-d+1)$-th column, then deleting the first row and the
$(n-d+1)$-th column (corresponding to the $b_{m-d,1}+b_{m-d,2}\imu$
term) in the corresponding matrix in \eqref{eq:abcond2-complex}.
Then, Eq.\ \eqref{eq:abcond2-complex-2} becomes as
\[
(N'_1+N'_2\imu)(\bm{v}'_1+\bm{v}'_2\imu)=\bm{0},
\]
with
\begin{equation}
  \label{eq:n1n2v1v2-monic}
  \begin{split}
  \setlength{\arraycolsep}{1pt}
  N'_1 &=
  \setlength{\arraycolsep}{1pt}
  \begin{pmatrix}
    \tilde{f}_{m-1,1}-\tilde{g}_{n-1,1} & 1 & & & 1 & &  \\
    \vdots & \tilde{f}_{m-1,1} & \ddots &  & \tilde{g}_{n-1,1} & \ddots & \\ 
    \tilde{f}_{0,1}-\tilde{g}_{n-m,1} & \vdots & \ddots & 1 &
    \vdots & \ddots & 1 \\
    & \tilde{f}_{0,1} & & \tilde{f}_{m-1,1} & \tilde{g}_{0,1} & &
    \tilde{g}_{n-1,1} \\
    & & \ddots & \vdots & & \ddots & \vdots \\
    & & & \tilde{f}_{0,1} & & & \tilde{g}_{0,1}
  \end{pmatrix}
  ,
  \\
  N'_2 &=
  \setlength{\arraycolsep}{1pt}
  \begin{pmatrix}
    \tilde{f}_{m-1,2}-\tilde{g}_{n-1,2} & 1 & & & 1 & &  \\
    \vdots & \tilde{f}_{m-1,2} & \ddots &  & \tilde{g}_{n-1,2} & \ddots & \\ 
    \tilde{f}_{0,2}-\tilde{g}_{n-m,2} & \vdots & \ddots & 1 &
    \vdots & \ddots & 1 \\
    & \tilde{f}_{0,2} & & \tilde{f}_{m-1,2} & \tilde{g}_{0,2} & &
    \tilde{g}_{n-1,2} \\
    & & \ddots & \vdots & & \ddots & \vdots \\
    & & & \tilde{f}_{0,2} & & & \tilde{g}_{0,2}
  \end{pmatrix}
  ,
  \\
  \bm{v}'_1 &= {}^t(a_{n-d,1},\ldots,a_{0,1},b_{m-d-1,1},\ldots,b_{0,1}),
  \\
  \bm{v}'_2 &=
  {}^t(a_{n-d,2},\ldots,a_{0,2},b_{m-d-1,2},\ldots,b_{0,2}).
  \end{split}
\end{equation}

\paragraph{Changes on the Settings in the Minimization Problem}

In solving the minimization problem, we substitute the variables
\begin{multline*}
  (
  \tilde{f}_{m-1,1},\ldots,\tilde{f}_{0,1},
  \tilde{g}_{n-1,1},\ldots,\tilde{g}_{0,1},
  \tilde{f}_{m-1,2},\ldots,\tilde{f}_{0,2},
  \tilde{g}_{n-1,2},\ldots,\tilde{g}_{0,2},
  \\
  a_{n-d,1},\ldots,a_{0,1},
  b_{m-d-1,1},\ldots,b_{0,1},
  a_{n-d,2},\ldots,a_{0,2},
  b_{m-d-1,2},\ldots,b_{0,2}
  )
\end{multline*}
as $\bm{x}=(x_1,\ldots,x_{4(m+n-d+1)})$, instead of
\eqref{eq:var0-complex}.  As a consequence, in contrast to
\eqref{eq:objective-complex2}, the objective function $f(\bm{x})$
becomes as
\begin{equation}
  \begin{split}
  f(\bm{x})
  =
  &
  (x_1-f_{m-1,1})^2+\cdots+(x_{m}-f_{0,1})^2
  \\
  &
  +
  (x_{m+1}-g_{n-1,1})^2+\cdots+(x_{m+n}-g_{0,1})^2
  \\
  &
  +
  (x_{m+n+1}-f_{m-1,2})^2+\cdots+(x_{2m+n}-f_{0,2})^2
  \\
  &
  +
  (x_{2m+n+1}-g_{n-1,2})^2+\cdots+(x_{2(m+n)}-g_{0,2})^2.
  \label{eq:objective-complex2-monic}
  \\
  \end{split}
\end{equation}
The constraints becomes as follows.  Now, Eq.\
\eqref{eq:abcond2-complex-3} becomes as 
\begin{equation}
  \label{eq:abcond2-complex-3-monic}
  \begin{pmatrix}
    N'_1 & -N'_2 \\
    N'_2 & N'_1
  \end{pmatrix}
  \begin{pmatrix}
    \bm{v}'_1 \\ \bm{v}'_2
  \end{pmatrix}
  =
  \bm{0},
\end{equation}
with $N'_1$, $N'_2$, $\bm{v}'_1$ and $\bm{v}'_2$ are defined as in 
\eqref{eq:n1n2v1v2-monic}.  Furthermore, the constraint for the
coefficients in $A(x)$ and $B(x)$ as in
\eqref{eq:abconstraint-complex} now becomes as
\begin{multline}
  \label{eq:abconstraint-complex-monic}
  \|A(x)\|_2^2+\|B(x)\|_2^2
  =(2a_{n-d,1}^2+\cdots+a_{0,1}^2)+(b_{m-d-1,1}^2+\cdots+b_{0,1}^2)
  \\
  +(2a_{n-d,2}^2+\cdots+a_{0,2}^2)+(b_{m-d-1,2}^2+\cdots+b_{0,2}^2)-1=0.
\end{multline}
Then, by the same way we have constructed
\eqref{eq:abcond2-complex-4}, we put
\eqref{eq:abcond2-complex-3-monic} and
\eqref{eq:abconstraint-complex-monic} together as
\begin{equation}
  \label{eq:abcond2-complex-4-monic}
  \begin{pmatrix}
    {}^t\bm{v}'_1 & {}^t\bm{v}'_2 & a_{n-d,1}^2+a_{n-d,2}^2-1 \\
    N'_1 & -N'_2 & \bm{0} \\
    N'_2 & N'_1 & \bm{0}
  \end{pmatrix}
  \begin{pmatrix}
    \bm{v}'_1 \\ \bm{v}'_2 \\ 1
  \end{pmatrix}
  =
  \bm{0},
\end{equation}
and we obtain the constraint $\bm{q}(\bm{x})=\bm{0}$ as
\begin{equation}
  \bm{q}(\bm{x})
  =
  {}^t(q_1(\bm{x}), \ldots, q_{2(m+n-d)+1}(\bm{x}))
  =
  \bm{0},
  \label{eq:constraint-complex2-monic}
\end{equation}
where $q_j(\bm{x})$ corresponds to the $j$-th row of matrix-vector
product in \eqref{eq:abcond2-complex-4-monic}.

\paragraph{Changes on the Initial Values}

Let $N'=U\,\Sigma\,{}^tV$ be the SVD of 
$
N'=
\begin{pmatrix}
  N'_1 & -N'_2 \\
  N'_2 & N'_1
\end{pmatrix}
$, with
\[
\begin{split}
  V &= (\bm{v}_1,\ldots,\bm{v}_{2(m+n-2d+1)}),\\
  \bm{v}_{2(m+n-2d+1)} &=
  {}^t(
  \bar{a}_{n-d,1},\ldots,\bar{a}_{0,1},
  \bar{b}_{m-d-1,1},\ldots,\bar{b}_{0,1},
  \\
  &\qquad
  \bar{a}_{n-d,2},\ldots,\bar{a}_{0,2},
  \bar{b}_{m-d-1,2},\ldots,\bar{b}_{0,2}
  ).
\end{split}
\]
Then, in contrast to \eqref{eq:appgcd-init-complex}, the initial value
becomes as
\begin{multline}
  \label{eq:appgcd-init-complex-monic}
  \bm{x}_0 = 
  (
  f_{m-1,1},\ldots,f_{0,1},
  g_{n-1,1},\ldots,g_{0,1},
  f_{m-1,2},\ldots,f_{0,2},
  g_{n-1,2},\ldots,g_{0,2},
  \\
  \bar{a}_{n-d,1},\ldots,\bar{a}_{0,1},
  \bar{b}_{m-d-1,1},\ldots,\bar{b}_{0,1},
  \bar{a}_{n-d,2},\ldots,\bar{a}_{0,2},
  \bar{b}_{m-d-1,2},\ldots,\bar{b}_{0,2}
  ).
\end{multline}

\paragraph{The Algorithm}
Summarizing discussions in the above, for preserving $\tilde{F}(x)$
and $\tilde{G}(x)$ to be monic, we modify Algorithm~\ref{alg:gpgcd} as
follows. 
\begin{alg}[GPGCD preserving monicity, with complex coefficients]
  \label{alg:gpgcd-complex-monic}
  Change Steps 1 and 2 in Algorithm~\ref{alg:gpgcd} as follows.
  \begin{description}
  \item[Step 1] \textbf{[Setting the initial values]} Set the initial values
    $\bm{x}_0$ as in \eqref{eq:appgcd-init-complex-monic}.
  \item[Step 2] \textbf{[Iteration]} Solve the minimization problem of
    $\bar{f}(\bm{x})=\frac{1}{2}(\bm{x})$, subject to
    $\bm{q}(\bm{x})=\bm{0}$, with $f(\bm{x})$ and $\bm{q}(\bm{x})$
    defined as in \eqref{eq:objective-complex2-monic} and
    \eqref{eq:constraint-complex2-monic}, respectively, as Step 2 in
    Algorithm~\ref{alg:gpgcd}.
  \end{description}
\end{alg}

\subsubsection{Running Time Analysis}

We give an analysis for running time of Algorithm~\ref{alg:gpgcd} with
employing the modified Newton method.

In Step 1, we set the initial values by the SVD. Since the dimension
of subresultant matrix is $O(m+n-d)$, running time in this step
becomes $O((m+n-d)^3)$.

In Step 2, we estimate running time just for one iteration, since the
number of iterations for convergence of solution may vary depending on
the given problem. This step essentially depends on solving the linear
system \eqref{eq:modnewton} with the Jacobian matrix $J_q(\bm{x}_k)$
defined as in \eqref{eq:jacobian-real} for the real coefficient case
or in \eqref{eq:jacobian-complex} for the complex coefficient case.
In both cases, dimension of $J_q(\bm{x}_k)$ is $O(m+n-d)$, 
thus we can estimate running time for solving the linear system
\eqref{eq:modnewton} as $O((m+n-d)^3)$.

Step 3 depends on calculating the least square solution of the linear
system as in Section~\ref{sec:correct} whose running time becomes as
$O((m+n-d)^3)$.

As a consequence, we can estimate running time of
Algorithm~\ref{alg:gpgcd} as $O((m+n-d)^3)$ times the number of
iterations for finding a GCD.

\subsection{Examples}

Now we show examples of Algorithm~\ref{alg:gpgcd} in the case of the
real coefficients (more comprehensive experiments are presented in the
next section).

Note that, for the minimization method, we have employed a modified
Newton method (Algorithm~\ref{alg:gpnewton}).  Computations in Example
\ref{ex:kar-lak1998} have been executed on a computer algebra system
Mathematica 6 with hardware floating-point arithmetic, while those in
Examples \ref{ex:small-lc} and \ref{ex:large-lc} have been executed on
another computer algebra system Maple 15 with \verb|Digits=10|.

\begin{exmp}
  \label{ex:kar-lak1998}
  This example is given by
  \citet{kar-lak1998},
  followed by
  \citet{kal-yan-zhi2007}.  Let
  $F(x),G(x)\in\R[x]$ be
  \[
  \begin{split}
    F(x) &= x^2 - 6 x + 5 = (x-1)(x-5), \\
    G(x) &= x^2 - 6.3 x + 5.72 = (x-1.1)(x-5.2),
  \end{split}
  \]
  and find $\tilde{F}(x),\tilde{G}(x)\in\R[x]$ which have the GCD of
  degree $1$, namely $\tilde{F}(x)$ and $\tilde{G}(x)$ have one common
  zero. 

  \textbf{Case 1}: The leading coefficient can be perturbed.  Applying
  Algorithm~\ref{alg:gpgcd} to $F$ and $G$, with $d=1$ and
  $\varepsilon=1.0\times 10^{-8}$, after $7$ iterations, we obtain the
  polynomials $\tilde{F}$ and $\tilde{G}$ as
  \[
  \begin{split}
    \tilde{F}(x) &= 0.985006 x^2 - 6.00294 x + 4.99942, \\ 
    \tilde{G}(x) &= 1.01495 x^2 - 6.29707 x + 5.72058,
  \end{split}
  \]
  with perturbations as
  $\sqrt{\|\tilde{F}-F\|_2^2+\|\tilde{G}-G\|_2^2}=0.0215941$
  and the
  common zero of $\tilde{F}(x)$ and $\tilde{G}(x)$ as
  $x=5.09890419203$.  In
  \citet{kal-yan-zhi2007}, the calculated perturbations obtained is
  $\sqrt{0.0004663}=0.021594$
  with the common zero as $x=5.09890429$.
  \citet{kar-lak1998} only give an example without
  perturbations on the leading coefficients.

  \textbf{Case 2}: The leading coefficient cannot be perturbed.
  Applying Algorithm~\ref{alg:gpgcd} (preserving monicity) with the
  same arguments as in Case 1, after $7$ iterations, we obtain the
  polynomials $\tilde{F}$ and $\tilde{G}$ as
  \[
  \begin{split}
    \tilde{F}(x) &= x^2 - 6.07504 x + 4.98528, \\ 
    \tilde{G}(x) &= x^2 - 6.22218 x + 5.73527,
  \end{split}
  \]
  with perturbations as
  $\sqrt{\|\tilde{F}-F\|_2^2+\|\tilde{G}-G\|_2^2}=0.110164$ And the
  common zero of $\tilde{F}(x)$ and $\tilde{G}(x)$ as
  $x=5.0969464650$.  In
  \citet{kal-yan-zhi2007}, the calculated perturbations obtained is
  $\sqrt{0.01213604583}=0.110164$
  with the common zero as $x=5.0969478$.  In
  \citet{kar-lak1998}, the calculated perturbations
  obtained is
  $\sqrt{0.01213605293}=0.110164$
  with the common zero as $x=5.096939087$.
\end{exmp}

The next examples, originally by
\citet{san-sas2007},
are ill-conditioned ones with the small or large leading coefficient
GCD.
\begin{exmp}[A small leading coefficient problem
  {\cite[Example~4]{san-sas2007}}]
  \label{ex:small-lc}
  Let $F(x)$ and $G(x)$ be
  \[
  \begin{split}
    F(x) &= (x^4 + x^2 + x + 1)(0.001x^2+x+1), \\
    G(x) &= (x^3 + x^2 + x + 1)(0.001x^2+x+1).
  \end{split}
  \]
  Applying Algorithm~\ref{alg:gpgcd} to $F$ and $G$, with $d=2$ and
  $\varepsilon=1.0\times10^{-8}$, after $1$ iteration, we obtain the
  polynomials $\tilde{F}$, $\tilde{G}$ and $H$ as
 \[
 \begin{split}
   \tilde{F}(x) &\simeq F(x), \quad \tilde{G}(x)\simeq G(x),\\
   H(x) &= 0.001x^2 + 0.9999999936 x + 0.9999999936,
 \end{split}
 \]
  with
  $\sqrt{\|\tilde{F}-F\|_2^2+\|\tilde{G}-G\|_2^2}=8.485281374\times
  10^{-12}$. 
\end{exmp}
\begin{exmp}[A large leading coefficient problem
  {\cite[Example~5]{san-sas2007}}] 
  \label{ex:large-lc}
  Let $F(x)$ and $G(x)$ be
  \[
  \begin{split}
    F(x) &= (x^6 - 0.00001 (0.8 x^5 + 3 x^4 - 4 x^3 - 4 x^2 - 5 x +
    1)) \cdot C(x), \\
    G(x) &= (x^5+x^4+x^3-0.1x^2+1)\cdot C(x),
  \end{split}
  \]
  with $C(x)=x^2+0.001$.  Applying Algorithm~\ref{alg:gpgcd} to $F$
  and $G$, with $d=2$ and $\varepsilon=1.0\times 10^{-8}$, after $1$
  iteration, we obtain the polynomials $\tilde{F}$, $\tilde{G}$ and
  $H$ as
 \[
 \begin{split}
   \tilde{F}(x) &\simeq F(x), \quad \tilde{G}(x)\simeq G(x),\\
   H(x) &= x^2 + 1.548794164\times 10^{-16} x + 0.001,
 \end{split}
 \]
  with
  $\sqrt{\|\tilde{F}-F\|_2^2+\|\tilde{G}-G\|_2^2}=1.735004369\times
  10^{-14}$.  
\end{exmp}

\section{Experiments}
\label{sec:exp}

We have implemented our GPGCD method (Algorithm~\ref{alg:gpgcd}) on a
computer algebra system Maple\footnote{The implementation is available
  at Project Hosting on Google Code \cite{gpgcd0.2}.} and carried out
the following tests:
\begin{enumerate}
\item \label{item:test-small} (Section \ref{sec:test-gp}) Comparison
  of performance of the gradient-projection method
  (Algorithm~\ref{alg:gp}) and the modified Newton method
  (Algorithm~\ref{alg:gpnewton}) on randomly generated polynomials
  with approximate GCD,
\item \label{item:test-large} (Section \ref{sec:test-appgcd})
  Comparison of performance of the GPGCD method with a method based on
  the structured total least norm (STLN) method by
  \citet{kal-yan-zhi2006} and the UVGCD method by \citet{zen2011} on
  large sets of randomly-generated polynomials with approximate GCD,
\item \label{item:test-zeng} (Section \ref{sec:test-zeng}) Comparison
  of performance of the GPGCD method with the STLN-based method and
  the UVGCD method on ill-conditioned polynomials and other test cases
  by \citet{zen2011} and \citet{bin-boi2010}.
\end{enumerate}
Note that, in Test~\ref{item:test-large}, we have tested both the
cases of the real and the complex coefficients, while, in the other
tests, we have tested only the case of the real coefficients.

In Tests \ref{item:test-small} and \ref{item:test-large}, we have
generated random polynomials with GCD then added noise, as follows.
First, we have generated a pair of monic polynomials $F_0(x)$ and
$G_0(x)$ of degrees $m$ and $n$, respectively, with the GCD of degree
$d$.  The GCD and the prime parts of degrees $m-d$ and $n-d$ are
generated as monic polynomials and with random coefficients
$c\in[-10,10]$ of floating-point numbers.  For noise, we have
generated a pair of polynomials $F_{\mathrm{N}}(x)$ and
$G_{\mathrm{N}}(x)$ of degrees $m-1$ and $n-1$, respectively, with
random coefficients as the same as for $F_0(x)$ and $G_0(x)$.  Then,
we have defined a pair of test polynomials $F(x)$ and $G(x)$ as
\[
  F(x)
  = F_0(x)+\frac{e_F}{\|F_{\mathrm{N}}(x)\|_2}F_{\mathrm{N}}(x),
  \quad
  G(x)
  = G_0(x)+\frac{e_G}{\|G_{\mathrm{N}}(x)\|_2}G_{\mathrm{N}}(x),
\]
respectively, scaling the noise such that the $2$-norm of the noise
for $F$ and $G$ is equal to $e_F$ and $e_G$, respectively.  In the
present test, we set $e_F=e_G=0.1$. (See also the notes in
Section~\ref{sec:test-appgcd}.)

The tests have been carried out on Intel Core2 Duo Mobile Processor
T7400 (in Apple MacBook ``Mid-2007'' model) at $2.16$ GHz with RAM
2GB, under Mac OS X 10.6.  All the tests have been carried out on Maple
15 with \verb|Digits=15| executing hardware floating-point arithmetic.

\subsection{Test~\ref{item:test-small}: Comparison of the
  Gradient-Projection Method and the Modified Newton Method}
\label{sec:test-gp}

In this test, we have compared performance of the gradient-projection
method (Algorithm~\ref{alg:gp}) and a modified Newton method
(Algorithm~\ref{alg:gpnewton}), only in the case of the real
coefficients.  For every example, we have generated one random test
polynomial as in the above, and we have applied
Algorithm~\ref{alg:gpgcd} (preserving monicity) with $u=100$ and
$\varepsilon=1.0\times 10^{-8}$.

Table~\ref{tab:gp} shows the result of the test: $m$ and $n$ denotes
the degree of a tested pair $F$ and $G$, respectively, and $d$ denotes
the degree of approximate GCD; ``Perturbation'' is the perturbation of
the perturbed polynomials from the initial inputs, calculated as
\begin{equation}
  \label{eq:perturbation}
  \sqrt{\|\tilde{F}-F\|_2^2+\|\tilde{G}-G\|_2^2},
\end{equation}
where ``$aeb$'' with $a$ and $b$ as numbers denotes $a\times 10^{b}$;
``\#Iterations'' is the number of iterations; ``Time'' is computing
time in seconds.  The columns with ``Alg.~\ref{alg:gp}'' and
``Alg.~\ref{alg:gpnewton}'' are the data for Algorithm~\ref{alg:gp}
(the gradient-projection method) and Algorithm~\ref{alg:gpnewton} (the
modified Newton method), respectively.  Note that, the
``Perturbation'' is a single column since both algorithms give almost
the same values in each examples.
\begin{table*}
  \centering
  \begin{tabular}{|c|c|c|c|c|c|c|c|}
    \hline
    Ex. & $m,n$ & $d$ & Perturbation &
    \multicolumn{2}{c|}{\#Iterations} & \multicolumn{2}{c|}{Time (sec.)}  \\
    \cline{5-8}
    & & &  \eqref{eq:perturbation} & Alg.~\ref{alg:gp} &
    Alg.~\ref{alg:gpnewton} & 
    Alg.~\ref{alg:gp} & Alg.~\ref{alg:gpnewton} \\
    \hline
    1 & $10,10$ & $5$ & $4.25e\!-\!2$ & $3$ & $4$ & $0.10$ & $0.04$ \\
    2 & $20,20$ & $10$ & $6.86e\!-\!2$ & $3$ & $4$ & $0.17$ & $0.11$ \\
    3 & $40,40$ & $20$ & $6.80e\!-\!2$ & $4$ & $5$ & $0.61$ & $0.16$ \\
    4 & $60,60$ & $30$ & $7.24e\!-\!2$ & $3$ & $4$ & $0.69$ & $0.23$ \\
    5 & $80,80$ & $40$ & $5.06e\!-\!2$ & $3$ & $4$ & $1.41$ & $0.41$ \\
    6 & $100,100$ & $50$ & $7.26e\!-\!2$ & $3$ & $4$ & $2.21$ & $0.76$ \\
    \hline
  \end{tabular}
  \caption[Test results comparing the gradient-projection method
  and the modified Newton method.]{Test results comparing the
    gradient-projection method and the modified Newton method; see
    Section~\ref{sec:test-gp} for details.} 
  \label{tab:gp}
\end{table*}

We see that, in all the test cases, the number of iterations of the
gradient-projection method (Algorithm~\ref{alg:gp}) is equal to $3$ or
$4$, which is smaller than that of the modified Newton method
(Algorithm~\ref{alg:gpnewton}) which is equal to $4$ or $5$.  However,
an iteration in Algorithm~\ref{alg:gp} includes solving a linear
system at least twice: once in the \textit{projection step} (Step 2)
and at least once in the \textit{restoration step} (Step 3); whereas
an iteration in Algorithm~\ref{alg:gpnewton} includes that only once.
Thus, total number of solving a linear system in
Algorithm~\ref{alg:gpnewton} is about a half of that in
Algorithm~\ref{alg:gp}.  Furthermore, computing time shows that the
modified Newton method runs approximately twice as fast as the
gradient projection method.  Therefore, we adopt
Algorithm~\ref{alg:gpnewton} as the method of minimization in the
GPGCD method (Algorithm~\ref{alg:gpgcd}).

\subsection{Test~\ref{item:test-large}: Tests on Large Sets of
  Randomly-generated Polynomials}
\label{sec:test-appgcd}

In this test, we have compared Algorithm~\ref{alg:gpgcd} with a method
based on the structured total least norm (STLN) method
(\cite{kal-yan-zhi2006}) and the UVGCD method (\cite{zen2008}), using
their implementation for the \textit{Maple}, in the both cases of the real and
the complex coefficients.  In our implementation of
Algorithm~\ref{alg:gpgcd}, we have chosen the modified Newton method
(Algorithm~\ref{alg:gpnewton}) for minimization.  In the STLN-based
method, we have used their procedure \verb|R_con_mulpoly| and
\verb|C_con_mulpoly|, which calculates the approximate GCD of several
polynomials in $\R[x]$ and $\C[x]$, respectively. In the UVGCD
method, we have used their procedure \verb|uvgcd| for calculating
approximate GCD of polynomials in $\R[x]$ and $\C[x]$.

Note that, in this test, we have defined test polynomials satisfying
another requirement: to make sure that the input polynomials $F(x)$
and $G(x)$ do not have a GCD of degree exceeding $d$, we have adopted
only those satisfying that the smallest singular value of the $d$-th
subresultant matrix $N_d(F,G)$ (see \eqref{eq:subresmat}) is larger
than or equal to $1$.\footnote{Our previous test results
  (\cite[Section 5.2]{ter2009}) have shown that there were test cases
  (input polynomials with the real coefficients) in which the GPGCD
  method was not able to calculate an approximate GCD with
  sufficiently small magnitude of perturbations. After thorough
  investigation, we have found that such input polynomials
  accidentally have an approximate GCD of degree exceeding $d$.  Thus,
  in the present test, we have defined totally new test polynomials
  satisfying the above requirement, then none of such phenomena have
  been observed with the test.}

For every example, we have generated $100$ random test polynomials as
in the above.  In executing Algorithm~\ref{alg:gpgcd}, we have set
$u=200$ and $\varepsilon=1.0\times 10^{-8}$; in \verb|R_con_mulpoly|
and \verb|C_con_mulpoly|, we have set the tolerance $e=1.0\times
10^{-8}$; in \verb|uvgcd|, we have set the initial tolerance
$\delta=1.0\times 10^{-2}$ and have changed it until we have
obtained an approximate GCD of desired degree.

Tables \ref{tab:appgcd} and \ref{tab:appgcd-complex} show the results
of the test in the case of the real and the complex coefficients,
respectively: $m$ and $n$ denotes the degree of a pair $F$ and $G$,
respectively, and $d$ denotes the degree of approximate GCD.  The
columns with ``STLN'' are the data for the STLN-based method;
``UVGCD'' are the data for the UVGCD method; ``GPGCD'' are the data
for the GPGCD method (Algorithm~\ref{alg:gpgcd}).  ``Perturbation'',
``\#Iterations'' and ``Time'' are the same as those in
Table~\ref{tab:gp}, respectively. (Note that computing time for the
UVGCD method does not include the time for ``try and error''
calculations by changing the tolerance $\delta$: it is just for
successful calculations.)

\begin{table*}
  \centering
  \setlength{\tabcolsep}{2.5pt}
  \begin{tabular}{|c|c|c|c|c|c|c|c|c|c|c|}
    \hline
    Ex. & $m,n$ & $d$ & 
    \multicolumn{3}{c|}{Perturbation \eqref{eq:perturbation}}  &
    \multicolumn{3}{c|}{Time (sec.)} & 
    \multicolumn{2}{c|}{\#Iterations} \\
    \cline{4-11}
    & & & STLN & UVGCD & GPGCD & STLN & UVGCD & GPGCD &
    STLN & GPGCD \\
    \hline
    1 & $10,10$ & $5$ &
    $5.64e\!-\!2$ & $1.79e\!-\!1$  & $5.64e\!-\!2$ &
    $0.38$ & $0.64$ & $0.04$ &
    $4.46$ & $4.50$\\ 
    2 & $20,20$ & $10$ &
    $6.22e\!-\!2$ & $1.85e\!-\!1$ & $6.22e\!-\!2$ &
    $1.16$ & $0.86$ & $0.06$ &
    $4.40$ & $4.40$ \\ 
    3 & $30,30$ & $15$ &
    $6.65e\!-\!2$ & $1.87e\!-\!1$ & $6.65e\!-\!2$ &
    $2.43$ & $1.34$ & $0.10$ &
    $4.37$ & $4.46$ \\ 
    4 & $40,40$ & $20$ &
    $6.48e\!-\!2$ & $1.96e\!-\!1$ & $6.48e\!-\!2$ &
    $4.05$ & $2.09$ & $0.13$ &
    $4.11$ & $4.15$ \\  
    5 & $50,50$ & $25$ &
    $6.91e\!-\!2$ & $1.91e\!-\!1$ & $6.91e\!-\!2$ &
    $6.30$ & $3.34$ & $0.19$ &
    $4.03$ & $4.16$ \\ 
    6 & $60,60$ & $30$ &
    $6.75e\!-\!2$ & $1.94e\!-\!1$ & $6.75e\!-\!2$ &
    $9.09$ & $4.37$ & $0.26$ &
    $4.00$ & $4.18$ \\ 
    7 & $70,70$ & $35$ &
    $6.89e\!-\!2$ & $2.08e\!-\!1$ & $6.89e\!-\!2$ &
    $12.47$ & $5.71$ & $0.35$ &
    $3.96$ & $4.13$ \\ 
    8 & $80,80$ & $40$ &
    $6.78e\!-\!2$ & $1.91e\!-\!1$ & $6.78e\!-\!2$ &
    $16.95$ & $7.95$ & $0.44$ &
    $3.16$ & $4.11$ \\ 
    9 & $90,90$ & $45$ &
    $6.92e\!-\!2$ & $1.95e\!-\!1$ & $6.92e\!-\!2$ &
    $22.09$ & $10.20$ & $0.57$ &
    $3.96$ & $4.10$ \\ 
    10 & $100,100$ & $50$ &
    $6.98e\!-\!2$ & $1.95e\!-\!1$ & $6.98e\!-\!2$ &
    $27.48$ & $13.02$ & $0.69$ &
    $3.88$ & $4.09$ \\ 
    \hline
  \end{tabular}
  \caption[Test results for large sets of polynomials with approximate
  GCD, in the case of the real coefficients.]{Test results for large
    sets of polynomials with approximate GCD, in the case of the real
    coefficients; see Section~\ref{sec:test-appgcd} for details.} 
  \label{tab:appgcd}
\end{table*}
\begin{table*}
  \setlength{\tabcolsep}{2.5pt}
  \centering
  \begin{tabular}{|c|c|c|c|c|c|c|c|c|c|c|}
    \hline
    Ex. & $m,n$ & $d$ & 
    \multicolumn{3}{c|}{Perturbation \eqref{eq:perturbation}}  &
    \multicolumn{3}{c|}{Time (sec.)} & 
    \multicolumn{2}{c|}{\#Iterations} \\
    \cline{4-11}
    & & & STLN & UVGCD & GPGCD & STLN & UVGCD & GPGCD &
    STLN & GPGCD \\
    \hline
    1 & $10,10$ & $5$ &
    $5.92e\!-\!2$ & $1.54e\!-\!1$  & $5.92e\!-\!2$ &
    $1.58$ & $0.64$ & $0.11$ &
    $4.50$ & $4.46$\\ 
    2 & $20,20$ & $10$ &
    $6.40e\!-\!2$ & $1.41e\!-\!1$ & $6.40e\!-\!2$ &
    $5.34$ & $1.31$ & $0.20$ &
    $4.30$ & $4.30$ \\ 
    3 & $30,30$ & $15$ &
    $6.63e\!-\!2$ & $1.40e\!-\!1$ & $6.63e\!-\!2$ &
    $11.63$ & $2.12$ & $0.35$ &
    $4.21$ & $4.24$ \\ 
    4 & $40,40$ & $20$ &
    $6.61e\!-\!2$ & $1.34e\!-\!1$ & $6.61e\!-\!2$ &
    $21.57$ & $3.51$ & $0.55$ &
    $4.15$ & $4.13$ \\  
    5 & $50,50$ & $25$ &
    $6.86e\!-\!2$ & $1.48e\!-\!1$ & $6.86e\!-\!2$ &
    $34.23$ & $5.03$ & $0.83$ &
    $4.06$ & $4.10$ \\ 
    6 & $60,60$ & $30$ &
    $6.86e\!-\!2$ & $1.51e\!-\!1$ & $6.86e\!-\!2$ &
    $50.40$ & $7.39$ & $1.16$ &
    $4.02$ & $4.05$ \\ 
    7 & $70,70$ & $35$ &
    $6.94e\!-\!2$ & $1.41e\!-\!1$ & $6.94e\!-\!2$ &
    $69.54$ & $10.31$ & $1.56$ &
    $3.93$ & $4.05$ \\ 
    8 & $80,80$ & $40$ &
    $6.85e\!-\!2$ & $1.44e\!-\!1$ & $6.85e\!-\!2$ &
    $93.77$ & $14.01$ & $2.07$ &
    $3.91$ & $4.07$ \\ 
    9 & $90,90$ & $45$ &
    $6.84e\!-\!2$ & $1.52e\!-\!1$ & $6.84e\!-\!2$ &
    $122.97$ & $18.30$ & $2.65$ &
    $3.90$ & $4.04$ \\ 
    10 & $100,100$ & $50$ &
    $6.94e\!-\!2$ & $1.65e\!-\!1$ & $6.94e\!-\!2$ &
    $157.02$ & $23.72$ & $3.37$ &
    $3.86$ & $4.04$ \\ 
    \hline
  \end{tabular}
  \caption[Test results for large sets of polynomials with approximate
  GCD, in the case of the complex coefficients.]{Test results for
    large sets of polynomials with approximate GCD, in the case of the
    complex coefficients; see Section~\ref{sec:test-appgcd} for
    details.}  
  \label{tab:appgcd-complex}
\end{table*}

We see that the average of magnitude of perturbations by the GPGCD
method is as small as that by the STLN-based method, which is
approximately one-tenth as large as that by the UVGCD method.  For
computing time, the GPGCD method calculates approximate GCD very
efficiently, faster than the STLN-based method by approximately from
$10$ to $30$ times and the UVGCD method by approximately from $6$ to
$10$ times.

\begin{rem}
  \label{rem:kaltofen}
  In this experiment, we have compared our implementation designed for
  problems of two univariate polynomials against the implementation of
  the STLN-based method designed for multivariate multi-polynomial
  problems with additional linear coefficient constraints.  Kaltofen
  \cite{kal2009} has reported that they have tested their
  implementation for just two univariate polynomials with real
  coefficients (\cite{kal-yan-zhi2007}) on an example similar to ours
  with degree 100 and GCD degree 50, and it took (on a ThinkPad of 1.8
  GHz with RAM 1GB) 2 iterations and 9 seconds.  This result will give
  the reader some idea on efficiency of our method.
\end{rem}

\subsection{Test~\ref{item:test-zeng}: Tests for Ill-conditioned
  Polynomials and Other Cases}
\label{sec:test-zeng}

In this test, we have compared Algorithm~\ref{alg:gpgcd} with the
STLN-based method (\cite{kal-yan-zhi2006}) and the UVGCD method
(\cite{zen2008}) on some ill-conditioned polynomials and other test
cases by \citet{zen2011} and \citet{bin-boi2010}, as follows.

Note that we give the degree of approximate GCD in the STLN-based
method and the GPGCD method, while we give the tolerance $\delta$ then
the algorithm estimates the degree of approximate GCD in the UVGCD
method.  Also note that, in some tests in this section, we have
measured the relative error of approximate GCD from the given GCD
\eqref{eq:relativeerror} instead of the magnitude of perturbation
\eqref{eq:perturbation} because, in such cases, we have given test
polynomials with predefined (approximate) GCD and have intended to
observe ``nearness'' of the calculated approximate GCD from the
predefined one.

Throughout the tables in
this section, the columns with ``STLN'', ``UVGCD'', ``GPGCD'',
``Perturbation'', and ``Time'' the same as those in the above,
respectively.

\begin{exmp}
  \label{exmp:cyclotomic}
  An example of ill-conditioned polynomial by \citet[Test 1]{zen2011}.
  Let $n$ be an even positive number and $k=n/2$, and define
  $p_n=u_nv_n$ and $q_n=u_nw_n$, where
  \begin{equation}
    \label{eq:cyclotomic}
    \begin{split}
      u_n &= \prod_{j=1}^k[(x-r_1\alpha_j)^2+r_1^2\beta_j^2],
      \quad
      v_n = \prod_{j=1}^k[(x-r_2\alpha_j)^2+r_2^2\beta_j^2],
      \\
      w_n &= \prod_{j=k+1}^n[(x-r_1\alpha_j)^2+r_1^2\beta_j^2],
      \quad
      \alpha_j=\cos\frac{j\pi}{n},
      \quad
      \beta_j=\sin\frac{j\pi}{n},
    \end{split}
  \end{equation}
  for $r_1=0.5$ and $r_2=1.5$. The zeros of $p_n$ and $q_n$ lie on the
  circles of radius $r_1$ and $r_2$. We had the test for
  $n=6,\ldots,20$ increased by $2$.


  Table~\ref{tab:cyclotomic} shows the result of the test.
 ``Relative error of GCD'' is calculated by
  \begin{equation}
    \label{eq:relativeerror}
    \frac{\|\bar{u}_n(x)-u_n(x)\|_2}{\|u_n(x)\|_2},
  \end{equation}
  where $u_n$ is predefined GCD as shown in \eqref{eq:cyclotomic} and
  $\bar{u}_n$ is approximate GCD. In the table, (*1) indicates that
  the STLN-based method did not converge within 50 times of iterations
  which is a built-in threshold; (*2) indicates that the GPGCD method
  did not converge within 100 times of iterations.  We see that, in
  the GPGCD method as well as in the STLN-based method, the number of
  iterations increases and the accuracy of calculated approximate GCD
  decreases as $n$ increases.  On the other hand, the UVGCD method has
  better accuracy of approximate GCD for large $n$.

  \begin{table}
    \centering
    \begin{tabular}{|c|c|c|c|}
      \hline
      $n$ & \multicolumn{3}{c|}{Relative error of GCD
        \eqref{eq:relativeerror}} \\ 
      \cline{2-4}
      & STLN & UVGCD & GPGCD \\
      \hline
      $6$ & $1.04e\!-\!14$ & $4.60e\!-\!15$ & $3.68e\!-\!15$ \\
      $8$ & $3.98e\!-\!13$ & $7.90e\!-\!13$ & $4.30e\!-\!13$ \\
      $10$ & $1.08e\!-\!10$ & $7.89e\!-\!12$ & $1.08e\!-\!10$ \\
      $12$ & $2.87e\!-\!10$ & $2.95e\!-\!11$ & $2.94e\!-\!10$ \\
      $14$ & $3.10e\!-\!9$ & $3.65e\!-\!10$ & $3.14e\!-\!9$ \\
      $16$ & $6.22e\!-\!9$ (*1) & $3.83e\!-\!10$ & $8.00e\!-\!9$ \\
      $18$ & $1.38e\!-\!6$ (*1) & $9.68e\!-\!9$ & $1.36e\!-\!6$ \\
      $20$ & $6.95e\!-\!6$ (*1) & $1.21e\!-\!8$ & $7.11e\!-\!6$ (*2) \\
      \hline
    \end{tabular}
    \caption{Test results for test polynomials
      \eqref{eq:cyclotomic}. See Example~\ref{exmp:cyclotomic} for details.}
    \label{tab:cyclotomic}
  \end{table}
\end{exmp}

\begin{exmp}
  \label{exmp:multiplegcd}
  Another example of ill-conditioned polynomial by \citet[Test
  2]{zen2011}. Let 
  \begin{equation}
    \label{eq:multiplegcd}
    p(x)=\prod_1^{10}(x-x_j), \quad q(x)=\prod_1^{10}(x-x_j+10^{-j}),
    \quad x_j=(-1)^j(j/2),
  \end{equation}
  The zeros of $q$ have decreasing distances as $0.1$, $0.01$, \ldots,
  from those of $p$. We have tried to calculate an approximate GCD of
  degree $d$ from $1$ to $10$ increased by $1$.
  
  Tables~\ref{tab:multiplegcd-1} and \ref{tab:multiplegcd-2} show the
  result of the test.  In this test, we have measured perturbation
  \eqref{eq:perturbation} since $p$ and $q$ are pairwisely relatively
  prime in a rigorous sense.  Note that we have put the results for
  the UVGCD method in Table~\ref{tab:multiplegcd-2}, separated from
  those for the GPGCD and the STLN-based methods in
  Table~\ref{tab:multiplegcd-1}, because we have given the tolerance
  $\delta$ to obtain approximate GCD in the UVGCD method, while we
  have given the degree $d$ in the GPGCD and the STLN-based
  methods. In Table~\ref{tab:multiplegcd-1}, (*1) indicates that the
  STLN-based method did not converge within 50 times of iterations
  which is a built-in threshold.
  
  We see that, for $d\ge 6$, all the methods find approximate GCD with
  similar magnitude of perturbations.  However, for smaller value of
  $d$, the UVGCD method finds approximate GCD with considerably
  smaller magnitude of perturbations than those in the other methods,
  followed by the STLN-based method.
  \begin{table}[t]
    \centering
    \begin{tabular}{|c|c|c|}
      \hline
      $d$ & \multicolumn{2}{c|}{Perturbation \eqref{eq:perturbation}} \\
      \cline{2-3}
      & STLN & GPGCD \\
      \hline
      $1$ & $5.17e\!-\!1$ (*1) &  $3.21e3$\\
      $2$ & $6.95e\!-\!4$ (*1) &  $3.06e0$\\
      $3$ & $1.97e\!-\!5$ &  $1.26e0$\\
      $4$ & $2.89e\!-\!6$ &  $2.25e\!-\!1$\\
      $5$ & $5.28e\!-\!5$ &  $4.75e\!-\!1$\\
      $6$ & $2.15e\!-\!3$ &  $2.16e\!-\!3$\\
      $7$ & $8.34e\!-\!2$ &  $8.34e\!-\!2$ \\
      $8$ & $2.04e0$ & $2.04e0$\\
      $9$ & $4.70e1$ & $4.70e1$ \\
      $10$ & $7.73e2$ & $7.73e2$ \\
      \hline
    \end{tabular}
    \caption{Test results for test polynomials
      \eqref{eq:multiplegcd} with the STLN-based method and the GPGCD
      method. See Example~\ref{exmp:multiplegcd} for details.} 
    \label{tab:multiplegcd-1}
  \end{table}
  \begin{table}[t]
    \centering
    \begin{tabular}{|c|c|c|}
      \hline
      \multicolumn{3}{|c|}{UVGCD} \\
      \hline
      $\delta$ & $d$ & Perturbation \eqref{eq:perturbation} \\
      \hline
      $1.0e\!-\!11$ & $1$ & $8.02e\!-\!10$ \\
      $1.0e\!-\!10$ & $2$ & $3.27e\!-\!8$ \\
      $1.0e\!-\!9$ & $3$ & $6.03e\!-\!7$ \\
      $1.0e\!-\!8$ & $4$ & $1.99e\!-\!5$ \\
      $1.0e\!-\!7$ & $5$ & $3.45e\!-\!4$ \\
      $1.0e\!-\!6$ & $5$ & $3.45e\!-\!4$ \\
      $1.0e\!-\!5$ & $6$ & $9.61e\!-\!3$ \\
      $1.0e\!-\!4$ & $7$ & $1.79e\!-\!1$ \\
      $1.0e\!-\!3$ & $8$ & $3.18e0$ \\
      $1.0e\!-\!2$ & $8$ & $3.18e0$ \\
      $1.0e\!-\!1$ & $9$ & $5.00e1$ \\
      $1.0e0$ & $10$ & $8.40e2$ \\
      \hline
    \end{tabular}
    \caption{Test results for test polynomials
      \eqref{eq:multiplegcd} with the UVGCD method. See
      Example~\ref{exmp:multiplegcd} for details.} 
    \label{tab:multiplegcd-2}
  \end{table}
\end{exmp}

\begin{exmp}
  \label{exmp:largedegreegcd}
  An example with GCDs of large degree by \citet[Test 3]{zen2011}. Let
  \begin{equation}
    \label{eq:largedegreegcd}
    \begin{split}
      p_n &=u_nv, \quad q_n=u_nw, \\
      v(x) &=\sum_{j=0}^3x^j, \quad w(x)=\sum_{j=0}^3(-x)^j,
    \end{split}
  \end{equation}
  where $u_n(x)$ is a GCD defined as
  a polynomial of degree $n$ whose coefficients are random integers in
  the range $[-5,5]$ and $v(x)$ and $w(x)$ are fixed cofactors.

  Table~\ref{tab:largedegreegcd} shows the result of the test by
  measuring relative error of approximate GCD
  \eqref{eq:relativeerror}.  In this test, we have also measured
  computing time because the difference of it became large among the
  methods for large degree of approximate GCD.  We see that the UVGCD
  method calculates approximate GCD with the best accuracy, followed
  by the STLN-based method and the GPGCD method.  On the other hand,
  the GPGCD method is more efficient than the other methods.
  \begin{table}
    \centering
    \begin{tabular}{|c|c|c|c|c|c|c|}
      \hline
      $n$ & \multicolumn{3}{c|}{Relative error of GCD
        \eqref{eq:relativeerror}} & 
      \multicolumn{3}{c|}{Time (sec.)} \\
      \cline{2-7}
      & STLN & UVGCD & GPGCD & STLN & UVGCD & GPGCD\\
      \hline
      $50$ & $1.60e\!-\!15$ & $1.04e\!-\!16$ & $2.63e\!-\!15$ &
      $1.77$ & $0.22$ & $0.04$ \\
      $100$ & $1.16e\!-\!15$ & $1.59e\!-\!16$ & $4.41e\!-\!15$ &
      $8.17$ & $0.31$ & $0.06$ \\
      $200$ & $1.14e\!-\!15$ & $1.06e\!-\!16$ & $1.23e\!-\!14$ &
      $45.09$ & $0.83$ & $0.12$ \\
      $500$ & $1.35e\!-\!15$ & $1.37e\!-\!16$ & $1.84e\!-\!14$ &
      $552.09$ & $3.39$ & $0.64$ \\
      $1000$ & $1.42e\!-\!15$ & $1.69e\!-\!16$ & $5.30e\!-\!14$ &
      $4318.38$ & $18.66$ & $3.27$ \\
      \hline
    \end{tabular}
    \caption{Test results for test polynomials
      \eqref{eq:largedegreegcd}. See Example~\ref{exmp:largedegreegcd}
      for details.} 
    \label{tab:largedegreegcd}
  \end{table}
\end{exmp}

\begin{exmp}
  \label{exmp:multiplezeros}
  An example with multiple zeros of high multiplicities by
  \citet[Example 4.5]{bin-boi2010}. Let
  \begin{equation}
    \label{eq:multiplezeros}
    u_k(x)=(x^3+3x-1)(x-1)^k,\quad v_k(x)=u'(x),
  \end{equation}
  for positive integer $k$.  Note that the GCD of $u_k(x)$ and
  $v_k(x)$ is $w_k(x)=(x-1)^{k-1}$.

  Table~\ref{tab:multiplezeros} shows the result of the test. In the
  table, as in Example~\ref{exmp:cyclotomic}, (*1) indicates that
  the STLN-based method did not converge within 50 times of iterations
  which is a built-in threshold; (*2) indicates that the GPGCD method
  did not converge within 100 times of iterations.

  We see that, in the GPGCD method as well as in the STLN-based
  method, the number of iterations increases and the accuracy of
  calculated approximate GCD decreases for $k=35$ and $45$.  On the
  other hand, the UVGCD method calculates approximate GCD accurately
  for large $k$.

  \begin{table}
    \centering
    \begin{tabular}{|c|c|c|c|}
      \hline
      $k$ & \multicolumn{3}{c|}{Relative error of GCD 
        \eqref{eq:relativeerror}} \\ 
      \cline{2-4}
      & STLN & UVGCD & GPGCD \\
      \hline
      $15$ & $2.35e\!-\!13$ & $3.08e\!-\!15$ & $1.86e\!-\!12$ \\
      $25$ & $1.64e\!-\!11$ & $1.13e\!-\!14$ & $6.67e\!-\!11$ \\
      $35$ & $3.79e\!-\!10$ (*1) & $8.02e\!-\!15$ & $3.58e\!-\!9$
      (*2) \\
      $45$ & $4.23e\!-\!8$ (*1) & $1.13e\!-\!14$ & $1.78e\!-\!7$ (*2)
      \\ 
      \hline
    \end{tabular}
    \caption{Test results for test polynomials
      \eqref{eq:multiplezeros}. See Example~\ref{exmp:multiplezeros}
      for details.} 
    \label{tab:multiplezeros}
  \end{table}
\end{exmp}

\begin{exmp}
  \label{exmp:multiplezeros-2}
  Another example with multiple zeros of high multiplicities by
  \citet[Test 6]{zen2011}. Let
  \begin{equation}
    \label{eq:multiplezeros-2}
    \begin{split}
    p_{[m_1,m_2,m_3,m_4]}(x) &= (x-1)^{m_1}(x-2)^{m_2}(x-3)^{m_3}(x-4)^{m_4},
    \\
    q_{[m_1,m_2,m_3,m_4]}(x) &= \frac{d}{dx}p_{[m_1,m_2,m_3,m_4]}(x),
    \end{split}
  \end{equation}
  for nonnegative integers $m_1,\ldots,m_4$. Note that the GCD of
  $p_{[m_1,m_2,m_3,m_4]}(x)$ and $q_{[m_1,m_2,m_3,m_4]}(x)$ is
  $(x-1)^{m'_1}(x-2)^{m'_2}(x-3)^{m'_3}(x-4)^{m'_4}$ with
  $m'_j=\max\{m_j-1,0\}$ for $j=1,\ldots,4$.
  
  Table~\ref{tab:multiplezeros-2} shows the result of the test. In the
  table, as in Examples~\ref{exmp:cyclotomic} and
  \ref{exmp:multiplezeros}, (*1) indicates that the STLN-based method
  did not converge within 50 times of iterations which is a built-in
  threshold; (*2) indicates that the GPGCD method did not converge
  within 100 times of iterations.  Furthermore, (*3) indicates that
  the GPGCD method stopped abnormally because the solution of a linear
  system with the coefficient matrix (the Jacobian matrix) as shown in
  \eqref{eq:jacobian-real} became unexpectedly large.

  We see that, in the GPGCD method as well as in the STLN-based
  method, the number of iterations increases and the accuracy of
  calculated approximate GCD becomes almost meaningless for inputs of
  large degree.  On the other hand, the UVGCD method is quite stable
  (in the sense of convergence of the algorithm) and more accurate for
  calculating approximate GCD for those inputs.

  \begin{table}
    \centering
    \begin{tabular}{|c|c|c|c|}
      \hline
      $[m_1,m_2,m_3,m_4]$ & \multicolumn{3}{c|}{Relative error of GCD 
        \eqref{eq:relativeerror}} \\ 
      \cline{2-4}
      & STLN & UVGCD & GPGCD \\
      \hline
      $[2,1,1,0]$ & $1.11e\!-\!13$ & $9.42e\!-\!16$ & $2.83e\!-\!13$ \\
      $[3,2,1,0]$ & $7.33e\!-\!13$ & $3.31e\!-\!15$ & $8.23e\!-\!12$ \\
      $[4,3,2,1]$ & $2.35e\!-\!9$ & $2.95e\!-\!13$ & $2.68e\!-\!9$ \\
      $[5,3,2,1]$ & $1.89e\!-\!8$  & $3.38e\!-\!12$ & $5.56e\!-\!9$ \\ 
      $[9,6,4,2]$ & $4.72e\!-\!8$ (*1) & $5.31e\!-\!11$ & $6.05e\!-\!8$ (*2) \\ 
      $[20,14,10,5]$ & $5.06e\!-\!1$ (*1) & $3.13e\!-\!10$ & $9.98e\!-\!1$ (*2) \\ 
      $[80,60,40,20]$ & $1.0e0$ (*1)  & $1.08e\!-\!3$ & $1.0e0$ (*2) \\ 
      $[100,60,40,20]$ & $1.0e0$ (*1) & $2.16e\!-\!4$ & N/A (*3) \\ 
      \hline
    \end{tabular}
    \caption{Test results for test polynomials
      \eqref{eq:multiplezeros-2}. See Example~\ref{exmp:multiplezeros-2}
      for details.} 
    \label{tab:multiplezeros-2}
  \end{table}
\end{exmp}

\section{Concluding Remarks}
\label{sec:remark}

We have proposed an iterative method, based on the modified Newton
method which is a generalization of the gradient-projection method,
for calculating approximate GCD of univariate polynomials with the
real or the complex coefficients.

Our experiments comparing the GPGCD method with the STLN-based method
and the UVGCD method have discovered advantages and disadvantages of
these methods, as follows.  In the case that input polynomials already
have exact or approximate GCD, then the UVGCD method calculates the
approximate GCD with the best accuracy and relatively fast convergence
among them. On the other hand, in the case that the magnitude of
``noise'' is larger, then the magnitude of perturbations calculated by
the GPGCD method or the STLN-based method is smaller than that
calculated by the UVGCD method.  Furthermore, in such cases, the GPGCD
method has shown significantly better performance over the other
methods in its speed, by approximately up to $30$ times for the
STLN-based method and $10$ times for the UVGCD method, which seems to
be sufficiently practical.  Other examples have shown that the GPGCD
method properly calculates approximate GCD with small or large leading
coefficient.

Our result have shown that, in contrast to the STLN-based methods
which uses \textit{structure preserving} feature for matrix
computations, our simple method can achieve accurate and efficient
computation as or more than theirs in calculating approximate GCDs in
many examples. On the other hand, our result have also shown that our
method is less accurate than the UVGCD method especially in the case
the given polynomials lie sufficiently close to polynomials that have
a GCD in a rigorous sense.  These results suggest that there are some
opportunities for improvements of accuracy and/or efficiency in
calculating approximate GCDs with optimization strategies.

For the future research, the followings are of interest.
\begin{itemize}
\item Convergence analysis of the minimizations: showing global
  convergence of local method is difficult in general (see e.g.\
  \citet{blu-cuc-shu-sma1996}), as the original paper on the modified
  Newton method (\cite{tan1980}) only shows its stability by observing
  whether the Jacobian matrix of the constraint at a local minimal
  point has full-rank or not.  However, it may be possible to analyze
  local convergence property depending on condition on the initial
  point and/or local minimal point. (See also
  Remarks~\ref{rem:stepwidth} and \ref{rem:jacobian}).
\item Improvements on the efficiency: time complexity of our method
  depends on the minimization, or solving a system of linear equations
  in each iteration.  Thus, analyzing the structure of matrices might
  improve the efficiency in solving a linear system.
\item Comparison with other methods (approaches) for approximate GCD:
  from various points of view such as accuracy, stability, efficiency,
  and so on, comparison of our methods with other methods will reveal
  advantages and drawbacks of our method in more detail.
\end{itemize}
Other topics, such as generalization of our method to several input
polynomials, are also among our next problems, some of which are
currently under our investigation (\cite{ter2010b}).

\section*{Acknowledgments}
We thank Erich Kaltofen and Zhonggang Zeng for making their
implementations for approximate GCD available on the Internet, Erich
Kaltofen for providing experimental results in
Remark~\ref{rem:kaltofen}, and Victor Pan for pointing literature on
structured matrix computations. We also thank Takaaki Masui as well as
anonymous reviewers for carefully reading the manuscript and their
valuable suggestions that helped improve the paper.

This research was supported in part by the Ministry of Education,
Culture, Sports, Science and Technology of Japan, under Grant-in-Aid
for Scientific Research (KAKENHI) 19700004.

\section*{Bibliography}


\def\cprime{$'$}

\end{document}